\documentclass{amsart} 
\usepackage{amsmath,amsfonts,amscd,amssymb,epsfig,euscript,bm, cancel}
\usepackage{amsxtra,xcolor,titletoc} 
\usepackage{mathrsfs}
\usepackage{tikz}
\usepackage{pgf,tikz-cd}
\usetikzlibrary{decorations.markings}
\usepackage{caption}
\numberwithin{equation}{section}
\usepackage{pdfsync}
\theoremstyle{plain}
 \newtheorem{theorem}{Theorem}[section]
\theoremstyle{definition} 
 \newtheorem{remark}[theorem]{Remark}
\newtheorem*{theorem*}{Theorem}
\newtheorem{proposition}{Proposition}
\newtheorem{lemma}{Lemma}
\newtheorem{corollary}{Corollary}
\theoremstyle{definition}

\newcommand{\field}[1]{\ensuremath{\mathbb{#1}}}
\newcommand{\CC}{\field{C}}

\newcommand{\HH}{\field{H}}

\newcommand{\RR}{\field{R}}

\newcommand{\ZZ}{\field{Z}}

\DeclareMathOperator{\id}{id}

\DeclareMathOperator{\im}{Im}

\newcommand{\del}{\partial}
\newcommand{\vp}{\varphi}
\newcommand{\delb}{\bar\partial}

\newcommand{\R}{\mathcal{R}}

\newcommand{\g}{\gamma}
\newcommand{\si}{\sigma}

\newcommand{\tr}{\mathrm{tr}}

\newcommand{\bk}{\backslash}
\newcommand{\Ga}{\Gamma}

\newcommand{\la}{\langle}
\newcommand{\ra}{\rangle}

\newcommand{\vep}{\varepsilon}

\newcommand{\curly}[1]{\mathscr{#1}}
\newcommand{\cA}{\curly{A}}

\newcommand{\cE}{\curly{E}}

\newcommand{\cH}{\curly{H}}

\newcommand{\cK}{\curly{K}}

\newcommand{\cN}{\curly{N}}

\newcommand{\cP}{\curly{P}}
\newcommand{\cQ}{\curly{Q}}

\newcommand\smallS{
  \mathchoice
    {{\scriptstyle\mathcal{S}}}
    {{\scriptstyle\mathcal{S}}}
    {{\scriptscriptstyle\mathcal{S}}}
    {\scalebox{.4}[.4]{$\scriptscriptstyle\mathcal{S}$}}
  }
\makeatletter
\@namedef{subjclassname@2020}{\textup{2020} Mathematics Subject Classification}
\makeatother

\usepackage{hyperref}
\hypersetup{colorlinks,citecolor=blue,plainpages=false,hypertexnames=false}

\begin{document}
\title[Goldman form and stable vector bundles]{Goldman form, flat connections and stable vector bundles}
\author{Leon A. Takhtajan}

\address{Department of Mathematics,
Stony Brook University, Stony Brook, NY 11794 USA; 
\newline
Euler International Mathematical Institute, Pesochnaya Nab.~10, Saint Petersburg 197022 Russia}
\email{leontak@math.stonybrook.edu}
  \dedicatory{To the Memory of M.S. Narasimhan (1932-2021)}

\begin{abstract}
We consider the moduli space $\cN$ of stable vector bundles of degree $0$ over a compact Riemann
surface and the affine bundle $\cA\to\cN$ of flat connections. Following  the similarity between the Teichm\"{u}ller spaces
and the moduli of bundles, we introduce the analogue of the quasi-Fuchsian projective connections --- local holomorphic sections of $\cA$ --- that allow to pull back the Liouville symplectic form on $T^{*}\cN$ to $\cA$. We prove that  the pullback of the Goldman form to $\cA$ by the Riemann-Hilbert correspondence coincides with the pullback of the Liouville form. We also include a simple proof, in the spirit of Riemann bilinear relations, of the classic result --
the pullback of Goldman symplectic form to $\cN$ by the Narasimhan-Seshadri connection is the natural symplectic form on $\cN$, introduced by Narasimhan and
Atiyah \& Bott.
\end{abstract}
\maketitle

\tableofcontents
\section{Introduction}
There is a close similarity between the theory of projective connections on Riemann surfaces, and the theory of flat connections in vector bundles over a Riemann surface (more generally, connections with central curvature), and it is worthwhile to describe it in some detail. 

Namely, the main object in the first theory is a holomorphic affine bundle $\cP_{g}\to T_{g}$ of projective connections over the Teichm\"{u}ller space $T_{g}$ of compact Riemann surfaces of genus $g>1$. Bers simultaneous uniformization determines a family of global holomorphic sections $T_{g}\to\cP_{g}$, allowing to identify $\cP_{g}$ with the holomorphic cotangent bundle $T^{*}T_{g}$ of
$T_{g}$. These sections correspond to quasi-Fuchsian projective connections and are parameterized by the points in $T_{g}$. Remarkably, the pullback to $\cP_{g}$ of the canonical symplectic form on $T^{*}T_{g}$ --- the Liouville form ---  does not depend  on the choice of a quasi-Fuchsian section \cite{Loustau}, as it follows from an important property, called the quasi-Fuchsian reciprocity in \cite{McM, TT}. The monodromy of a projective connection determines a natural map of $\cP_{g}$ to the $\mathrm{PSL}(2,\CC)$-character variety, allowing to pull back the Goldman symplectic form to $\cP_{g}$. As stated in \cite{Kawai}, these two pullbacks give the same (up to a constant) symplectic form on $\cP_{g}$ (see \cite{Tak} for a direct proof).

The Teichm\"{u}ller space $T_{g}$ is also a symplectic manifold with symplectic form given by the K\"{a}hler form of the Weil-Petersson metric.
It is naturally isomorphic to the component of the $\mathrm{PSL}(2,\RR)$-character variety with the maximal Euler class  \cite{Goldman2}, 
and the pullback of the Goldman form to $T_{g}$ is (up to a constant) the Weil-Petersson symplectic form  \cite{Goldman1}. 

The celebrated Narasimhan-Seshadri theorem is an analogue of the Fuchsian uniformization of Riemann surfaces for the stable vector bundles over a compact Riemann surface, and the moduli space $\cN$ of
stable bundles of rank $n$ and degree $0$ naturally has a structure of the $\mathrm{U}(n)$-character variety. The analog of Goldman theorem in \cite{Goldman1} is Theorem \ref{E-S-G} --- a simple statement that  the
pullback of the Goldman form on the $\mathrm{U}(n)$-character variety is the Narasimhan-Atiyah-Bott symplectic form on the module space $\cN$. It is attributed to Goldman and is well-known to the experts. 
For convenience of the reader, we present a simple proof in the spirit of Riemann bilinear relations, which goes back to the classic work of Eichler and Shimura.

Similarly to the Teichm\"{u}ller theory, there is a holomorphic affine bundle 
$\cA\to\cN$ of zero curvature connections in the stable vector bundles, compatible with the complex structure. The Riemann-Hilbert correspondence --- the monodromy map --- maps $\cA$ to the $\mathrm{GL}(n,\CC)$-character variety and allows to pull back the Goldman form to $\cA$.  It is worth mentioning, though we will not use it in the paper, that according to the non-abelian Hodge correspondence (Hitchin-Simpson-Donaldson-Corlette), the $\mathrm{GL}(n,\CC)$-character variety is isomorphic to the  moduli space of stable Higgs bundles of rank $n$ and degree $0$ on a compact Riemann surfaces of genus $g>1$, and the image of $\mathcal{A}$ under the Riemann-Hilbert correspondence is dense.

The Narasimhan-Seshadri theorem provides a canonical section of the bundle $\cA\to\cN$, that establishes the real-analytic isomorphism\footnote{The non-abelian Hodge correspondence establishes a more general isomorphism between the moduli spaces of stable connections and of stable Higgs bundles.} $\cA\simeq T^{*}\cN$.
However, as opposed to the Teichm\"{u}ller theory, the bundle $\cA\to\cN$ has no global holomorphic sections to establish a complex-analytic isomorphism. Nevertheless, one can define a family of 
local holomorphic sections, which we call `quasi-unitary', and consider the family of local pullbacks of the holomorphic Liouville form on $T^{*}\cN$. In Proposition \ref{qu-rec} we prove an analogue of the quasi-Fuchsian reciprocity, the `quasi-unitary reciprocity', which leads to Proposition \ref{independence} that local pullbacks determine a global holomorphic symplectic form on $\cA$.  One can compare this symplectic form on $\cA$ with the pullback of the Goldman form by the Riemann-Hilbert correspondence, and according to Theorem \ref{main} they are the same (up to a constant).

It should be noted that implicitly the Goldman form was already present in the formulation of Theorem 21 in Gunning's 1967 lectures on vector bundles \cite{Gunning}, which was an analogue of the Riemann bilinear relations for the bundle-valued $1$-forms. Goldman form also appears in the theory of integrable systems. Namely, Krichever in \cite{Krichever} introduced a natural symplectic form on the space of Lax operators on a genus $g$ algebraic curve --- the space of meromorphic connections on stable vector bundles of rank $n$ and degree $ng$ over the curve. He showed that on the open subset of the moduli space, consisting of bundles characterized by their respective Tyurin parameters (see \cite{Krichever} for the details and references), the symplectic form on the space of Lax operators is the Liouville form on the cotangent bundle of the moduli space. Its explicit description in terms of the monodromy data is given in Sections 5-6 in \cite{Krichever}, 
and it is not difficult to show that formula (6.9) in \cite{Krichever} coincides with the Goldman symplectic form\footnote{Actually, Krichever's construction \cite{Krichever} is more general and includes the case of irregular singular points and corresponding Stokes' data, which is beyond the realm of algebraic geometry.}. In a recent paper \cite{Bertola}, Krichever's theorem (see \cite[Theorem 6.1]{Krichever}) was reformulated as Theorem  7.4, and rather complicated analytic tools were used for its derivation.

The purpose of the present paper is to describe the similarity between the moduli of curves and the moduli of bundles in a clear and simple way using basic analytic and algebraic geometry methods. In this sense, it complements the foundational papers \cite{NS2, AB, Donaldson} as well as \cite{ZT}.

Here is a more detailed content of the paper. In \S\,\ref{1} we remind necessary basic facts about the moduli space $\cN$ of stable vectors bundles. Specifically,  in \S\,\ref{1.1} we recall the Narasimhan-Seshadri theorem,
and in \S\,\ref{m space} describe the complex structure and the Narasimhan-Atiyah-Bott K\"{a}hler metric on $\cN$. In \S\,\ref{c-coordinates} we review introduced in \cite{ZT}  local complex coordinates on $\cN$. Their description uses differential equation \eqref{dbar} --- a vector bundle analogue of the Beltrami equation in the Teichm\"{u}ller theory, and these complex coordinates are analogous to the Bers coordinates on the Teichm\"{u}ller space. For convenience of the reader we include the proof of Proposition \ref{B-coordinates}. In \S\,\ref{Eichler} we review Eichler integrals for harmonic $(0,1)$-forms, used 
in Lemma \ref{a-lemma} in \S\,\ref{2.5} for the explicit solution of the infinitesimal form of the equation \eqref{dbar}. In \S\,\ref{ch-variety} we briefly discuss the Goldman symplectic form on the character variety, referring to  \cite{Goldman1, Tak} for the details. For convenience of the reader, in \S\,\ref{3-2} we prove Theorem \ref{th: E-S}, the Riemann bilinear relations for the matrix-valued $1$-forms expressed in terms of the Eichler-Shimura periods, and as a corollary obtain Theorem \ref{E-S-G}, that the pullback to $\cN$ of the Goldman form on the $\mathrm{U}(n)$-character variety is the Narasimhan-Atiyah-Bott symplectic form. 

Next, in \S\,\ref{zero-curv} we introduce the affine bundle $\cA\to\cN$ of zero curvature connections. In \S\,\ref{RH} we discuss the Riemann-Hilbert correspondence that maps fibers of $\cA$ to the set of equivalence classes of constant transition functions of a stable vector bundle, and in \S\,\ref{holo-rep} define the local families of bundles with constant transition functions that depend holomorphically on the complex coordinates. This construction is used in \S\,\ref{hol sections} to define, over each coordinate chart, local holomorphic sections of $\cA$, which we call quasi-unitary connections $\smallS_{\si}$, and in Lemma \ref{NS-qu} we express the difference 
between the Narasimhan-Seshadri section $\smallS_{NS}$ and $\smallS_{\si}$ as a $(1,0)$-form on $U$. In \S\,\ref{reciprocity} prove the Proposition \ref{qu-rec}, the quasi-unitary reciprocity, that this difference $\smallS_{NS}-\smallS_{\si}$ is $\del$-closed on $U$. This result is crucial for defining the pullback to $\cA$ of the Liouville symplectic form on $T^{*}\cN$ by using local quasi-unitary sections. Namely, in \S\S\,\ref{hol symplectic}-\ref{L-A}, following \cite{Loustau}, we prove Lemma \ref{unique} --- a criterion when pullbacks of $\omega_{L}$ to $\cA$ by two local sections give the same result, and Proposition \ref{independence} --- local pullbacks of $\omega_{L}$ determine a global symplectic form on $\cA$ that does not depend on a choice of quasi-unitary sections. In \S\,\ref{final} we prove our main result. Specifically, in \S\,\ref{differential} we explicitly compute the differential of the Riemann-Hilbert correspondence, and in \S\,\ref{pullback} we prove Theorem \ref{main} that pullbacks to $\cA$ of Liouville and Goldman symplectic forms are equal (up to a constant). Finally, in \S\,\ref{generalization} we discuss obvious generalizations to the case of stable bundles of arbitrary rank and degree, and to the case of parabolic vector bundles\footnote{One can also extend these results to the moduli spaces of stable Higgs bundles, which is beyond the scope of this paper.}. 

This paper could be written using the modern language of modular stacks, etc. However, since the material  is rather basic and the proofs are simple, we choose a `neoclassical' style to make the paper accessible to a wider audience.
\subsection*{Acknowledgments}I thank Igor Krichever for the explanation of his work on Lax operators on algebraic curves \cite{Krichever} and useful discussion, and Indranil Biswas and Dmitri Orlov for the remarks on non-existence of global holomorphic sections in \S\,\ref{hol sections}. I am grateful to Claudio Meneses for the valuable comments and suggestions for improving the presentation in the paper. I am especially grateful to the anonymous referee for the constructive remarks and suggestions. In particular, the referee advised to include the commutative diagram \eqref{sum-diagram}, which neatly summarizes results of the paper.

\section{Moduli space of flat vector bundles}\label{1}
\subsection{Narasimhan-Seshadri theorem} \label{1.1} Let $X$ be a compact Riemann surface of genus $g>1$, and let $E\to X$ be a holomorphic vector bundle of rank $n$ and degree $d$. Denote by 
$\mu(E)=d/n$ the slope of $E$. The bundle $E$ is said to be stable, if for every proper holomorphic subbundle $F$ of $E$ we have
$$\mu(F)<\mu(E).$$
In what follows we will consider stable bundles of degree $0$ --- flat bundles. The theorem of Narasimhan-Seshadri --- an analogue of the Fuchsian uniformization for vector bundles --- states that every stable flat bundle over a Riemann surface $X$ of genus $g>1$ arises from an irreducible unitary representation of 
the fundamental group $\pi_{1}$ of $X$. 

Namely, let $X\simeq\Gamma\bk\HH$,
where 
$$\HH=\{z=x+\sqrt{-1}y\in\CC : y>0\}$$ 
is the Lobachevsky (hyperbolic) plane, and $\Gamma$ is a co-compact Fuchsian group uniformizing the Riemann surface $X$, so that $\Gamma\simeq\pi_{1}$. 
Let $\rho:\Gamma\to \mathrm{U}(n)$ be an irreducible unitary representation of $\Gamma$, and let $\HH\times\CC^{n}\to\HH$ be the trivial bundle with the following $\Gamma$-action: 
$$(z,v)\mapsto (\gamma z,\rho(\gamma)v),\quad \text{where}\quad z\in\HH,\;v\in\CC^{n}\quad\text{and}\quad\gamma\in\Gamma.$$
Denote by $E_{\rho}$ the corresponding quotient bundle $\Gamma\bk(\HH\times\CC^{n})\to\Gamma\bk\HH\simeq X$ --- a holomorphic flat bundle over $X$. 
We have the following key result.

\begin{theorem*}[Narasimhan-Seshadri] A holomorphic vector bundle $E$ over $X$ of rank $n$ and degree $0$
is stable if and only if it is isomorphic to a bundle $E_{\rho}$, where $\rho$ is an irreducible unitary
representation of the group $\Gamma$. The bundles $E_{\rho_{1}}$ and $E_{\rho_{2}}$ are isomorphic if
and only if the representations $\rho_{1}$ and $\rho_{2}$ are equivalent.
\end{theorem*}

The standard Hermitian metric in $\CC^{n}$ defines a metric in the trivial bundle $\HH\times\CC^{n}\to\HH$. This metric is $\Gamma$-invariant and determines the Hermitian 
metric $h_{E}$ in the quotient bundle $E$, and the Hermitian metric $h_{\mathrm{End} \,E}$  in the endomorphism bundle $\mathrm{End} \,E$. Corresponding canonical connection $\nabla_{E}=d+A_{E}$ in the Hermitian bundle $(E,h_{E})$ --- a unique connection, compatible with the metric $h_{E}$ and with the 
holomorphic structure in $E$ --- is called the \emph{Narasimhan-Seshadri connection}. It is associated with the trivial connection $\nabla=d+ 0$ in the trivial bundle $\HH\times\CC^{n}\to\HH$, and has zero curvature. Conversely, it was proved by Donaldson \cite{Donaldson}, that
an indecomposable degree $0$ holomorphic vector bundle $E$ over $X$ is stable, if there exists a Hermitian metric  in $E$ such that the corresponding canonical connection has zero curvature.
 
The metric   $h_{\mathrm{End} \,E}$ determines a Hodge $\ast$-operator in the vector space $\Omega^{0,1}(X, \mathrm{End} \,E)$ of smooth $(0,1)$-forms on $X$ with values in $\mathrm{End} \,E$. The Hilbert space $\frak{H}^{0,1}(X, \mathrm{End} \,E)$ is the completion
of $\Omega^{0,1}(X, \mathrm{End} \,E)$ with respect to the Hodge inner product
\begin{equation}\label{Hodge-inner}
\la \mu,\nu\ra=\int_{X}\mu\wedge \ast \nu,
\end{equation}
where $\wedge$ is a composition of the exterior product of differential forms on $X$ and the fiber-wise trace map $\tr$ in $\mathrm{End}\, E$. 
Denote by $\cH^{0,1}(X,\mathrm{End}\, E)$ the zero eigenspace of the corresponding $\bar\del$-Laplace operator in $\frak{H}^{0,1}(X, \mathrm{End} \,E)$, the subspace of harmonic
$\mathrm{End} \,E$-valued $(0,1)$-forms on $X$, and by 
$$P: \frak{H}^{0,1}(X, \mathrm{End} \,E)\to\cH^{0,1}(X,\mathrm{End}\, E)$$ 
--- the corresponding orthogonal projection operator. 

In case $E=E_{\rho}$ the vector space $\Omega^{0,1}(X, \mathrm{End} \,E)$ is naturally identified with the space of smooth $\mathrm{End}\, \CC^{n}$-valued functions on $\HH$ satisfying
\begin{equation}\label{mu-transform}
\mu(\gamma z)\overline{\gamma'(z)}=\mathrm{Ad}\,\rho(\gamma)\mu(z)\quad\text{for all}\quad \gamma\in\Gamma,
\end{equation}
where $\mathrm{Ad}\,\rho$ is the adjoint representation of the group $\Gamma$ in the vector space $\mathrm{End}\, \CC^{n}$.  Correspondingly, $\cH^{0,1}(X,\mathrm{End}\, E)$ is the space of antiholomorphic functions $\mu(z)$ satisfying \eqref{mu-transform}. The Hodge  $\ast$-operator is $\ast \mu=-\sqrt{-1}\mu^{\ast}$, where
$\mu^{\ast}=\bar{\mu}^{\mathrm{t}}$ is the Hermitian conjugation of a $n\times n$-matrix $\mu$, and the Hodge inner product \eqref{Hodge-inner} becomes
\begin{equation*}
\la\mu, \nu\ra=\sqrt{-1}\iint_{F}\tr\left(\mu(z)\nu(z)^{\ast}\right)dz\wedge d\bar{z},
\end{equation*}
where $F$ is a fundamental domain for $\Gamma$ in $\HH$.

\subsection{The moduli space} \label{m space}The moduli space of
stable vector bundles of rank $n$ and degree $0$ over a compact Riemann surface $X$
of genus $g>1$ is the set $\cN$ of isomorphism classes $\{E\}$ of stables bundles $E$ over $X$. By the Narasimhan-Seshadri theorem, $\cN$ coincides with the $\mathrm{U}(n)$-character
variety 
$$\cK=\mathrm{Hom}_{0}(\pi_{1}, \mathrm{U}(n))/\mathrm{U}(n),$$
where it is understood that $\pi_{1}\simeq\Gamma$, and the subscript $0$ stands for the irreducible representations. The moduli space is a 
complex manifold of dimension $d=n^{2} (g - 1) + 1$ (see \cite{NS1,NS2}).  

Specifically, according
to the Kodaira-Spencer theory of deformation of complex structures\footnote{See \cite{NS1,NS2}) for the details in connection with vector bundles.}, 
the holomorphic tangent space $T_{E}\cN$ to the manifold $\cN$ at the point $\{E\}$ corresponding
to a stable bundle $E$ is identified with the complex vector space $\cH^{0,1}(X,\mathrm{End}\, E)$ of harmonic
$(0,1)$-forms. 
The holomorphic cotangent space  $T^{*}_{E}\cN$ is then identified with the complex vector space $\cH^{1,0}(X,\mathrm{End}\, E)$  of harmonic $ (1,0)$-forms, the space of Higgs fields. 
The natural pairing between these vector spaces is given by the integration:  
$$\cH^{1,0}(X,\mathrm{End} \,E)\otimes\cH^{0,1}(X,\mathrm{End} \,E)\ni\theta\otimes\mu\mapsto \int_{X}\theta\wedge\mu\in\CC.$$

The Hermitian inner product \eqref{Hodge-inner} in the vector spaces $\cH^{0,1}(X,\mathrm{End}\, E)$ determines a natural Hermitian metric $ds^{2}$ in $\cN$, introduced by Narasimhan \cite{Narasimhan} in 1969, and later by Atiyah and Bott \cite{AB}. This metric is an analogue of the Weil-Petersson metric in Teichm\"{u}ller spaces, and also is K\"{a}hler \cite{Narasimhan,AB}. We call it the  \emph{Narasimhan-Atiyah-Bott metric}, and denote by $\omega_{NAB}$ its symplectic form, the associated $(1,1)$-form on $\cN$,
$$\omega_{NAB}=-\frac{1}{2}\im ds^{2}.$$

\subsection{Complex coordinates} \label{c-coordinates} It was shown in \cite{ZT}, that is possible to describe the complex structure of the moduli space $\cN$ by introducing complex coordinates in a neighborhood of each point. These coordinates are analogous of the Bers coordinates in Teichm\"{u}ller spaces, and are convenient for performing local computations. 

Namely, realize a stable bundle $E$ as the quotient bundle $E_{\rho}$ corresponding to the irreducible unitary representation $\rho$  of the group $\Ga$, and $\mu\in\Omega^{0,1}(X,\mathrm{End}\, E)$ --- as $\mathrm{End} \,\CC^{n}$-function on $\HH$, satisfying \eqref{mu-transform}. The analog of the Beltrami equation in the Teichm\"{u}ller theory is the following $\bar\del$-problem
\begin{equation}\label{dbar}
\frac{\del}{\del\bar{z}} f(z)= f(z)\mu(z),\quad z\in\HH.
\end{equation}
The next result was proved in \cite{ZT}. For convenience of the reader, we present here its simple proof.
\begin{proposition}\label{B-coordinates}
If $\mu\in\cH^{0,1}(X,\mathrm{End}\, E)$ is sufficiently close to zero, then equation \eqref{dbar} has is a unique solution $f^{\mu}:\HH\to\mathrm{GL}(n,\CC)$ 
with the following properties.
\begin{itemize}
\item[(i)] For each $\g\in\Ga$
$$f^{\mu}(\g z) =\rho^{\mu}(\g)f^{\mu}(z)\rho(\g)^{-1},$$  
where $\rho^{\mu}$ is an irreducible representation of the group $\Ga$ in $\mathrm{U}(n)$.
\item[(ii)] For a fixed $z_{0}\in\HH$ (say $z_{0}=\sqrt{-1}$) the matrix $f^{\mu}(z_{0})$ is positive definite and $\det f^{\mu}(z_{0})=1$.
\end{itemize}
\end{proposition}

\begin{proof}
Since $\mu\in\cH^{0,1}(X,\mathrm{End}\, E)$ is an anti-holomorphic matrix-valued function on $\HH$, the ordinary differential equation 
\begin{equation}\label{ode}
\frac{ df}{d\bar{z}}= f(z)\mu(z)
\end{equation}
has a unique anti-holomorphic solution --- the $n\times n$ matrix-valued function $f_{\mu}(z)$ normalized by $f_{\mu}(z_{0})=I$, where $I$ is the identity matrix. It follows from \eqref{mu-transform} that
$$f_{\mu}(\g z) =\rho_{\mu}(\g)f_{\mu}(z)\rho(\g)^{-1}\quad\text{for all}\quad \g\in\Ga,$$ 
where $\rho_{\mu}:\Ga\to\mathrm{GL}(n,\CC)$ is a representation of the group $\Gamma$.
 For sufficiently small $\mu$ the quotient bundle
$E_{\rho_{\mu}}=\Ga\bk(\HH\times\CC^{n})$ associated with the representation $\rho_{\mu}$ is stable, so representation $\rho_{\mu}$ is irreducible. According to  the Narasimhan-Seshadri theorem,  there
is an irreducible representation $\rho^{\mu}: \Gamma\to\mathrm{U}(n)$ such that $E_{\rho_{\mu}}\simeq E_{\rho^{\mu}}$. In other words, there is holomorphic mapping
$g^{\mu}:\HH\to\mathrm{GL}(n,\CC)$, such that 
$$g^{\mu}(\g z) =\rho^{\mu}(\g)g^{\mu}(z)\rho_{\mu}(\g)^{-1}\quad\text{for all}\quad \g\in\Ga,$$
and we put $f^{\mu}(z)=g^{\mu}(z)f_{\mu}(z)$. 

The proof of uniqueness is also easy. Suppose $f_{1}(z)$ and $f_{2}(z)$ are two solutions of \eqref{dbar} satisfying property (i) with unitary representations $\rho_{1}$ and $\rho_{2}$. Then $h(z)=f_{1}(z)f_{2}(z)^{-1}$ is a holomorphic function satisfying $h(\gamma z)=\rho_{1}(\gamma)h(z)\rho_{2}(\gamma)^{-1}$. It determines an isomorphism $E_{\rho_{1}}\simeq E_{\rho_{2}}$, and by the Narasimhan-Seshadri theorem, $h(z)$ is a scalar
multiple of a unitary matrix; it follows form the property (ii) that $h(z)=I$.
\end{proof}

Now for each stable bundle $E=E_{\rho}$ choose a basis $\mu_{1},\dots,\mu_{d}$ in the vector space $\cH^{0,1}(X,\mathrm{End}\, E)$, and put
$\mu=\vep_{1}\mu_{1}+\cdots+\vep_{d}\mu_{d}$. According to Proposition \ref{B-coordinates}, this introduces complex coordinates $(\vep_{1},\dots,\vep_{d})$ in the coordinate chart at the point $\{E\}\in\cN$ --- a neighborhood $U$ of $\{E\}$ --- determined by the condition that $\mu$ is sufficiently close to zero. Indeed, for different $\mu$ representations $\rho^{\mu}$ are not equivalent, as it follows from the general Kodaira-Spencer theory. It can be also verified directly using Corollary \ref{cor-0}, as in \cite[Ch. VI D]{Ahlfors-2} for the Teichm\"{u}ller space. These coordinates transform holomorphically and endow $\cN$ with the structure of a complex manifold. 

Specifically, the differential of such coordinate 
transformation in the intersection of the coordinate charts at $\{E_{\rho}\}$ and $\{E_{\rho^{\mu}}\}$ is a linear mapping of vector spaces $\cH^{0,1}(X,\mathrm{End}\, E_{\rho})$ and $\cH^{0,1}(X,\mathrm{End}\, E_{\rho^{\mu}})$, explicitly given by the formula
$$\cH^{0,1}(X,\mathrm{End}\, E_{\rho})\ni\nu\mapsto P_{\mu}(\mathrm{Ad}f^{\mu}(\nu))\in\cH^{0,1}(X,\mathrm{End}\, E_{\rho^{\mu}}).$$
Here $P_{\mu}$ is the orthogonal projection operator onto $\cH^{0,1}(X,\mathrm{End}\, E_{\rho^{\mu}})$, and
$\mathrm{Ad}\,f^{\mu}$ is a fiberwise linear mapping $\mathrm{End}\, E_{\rho}\to\mathrm{End}\, E_{\rho^{\mu}}$, where $\mathrm{Ad}\,f^{\mu}(\nu)=f^{\mu}\cdot\nu\cdot(f^{\mu})^{-1}$.
Corresponding vector fields $\del/\del\vep_{i}$ on $U$ at each point $\{E_{\rho^{\mu}}\}\in U$ 
are given by
$$\left.\frac{\del}{\del\vep_{i}}\right|_{\mu}=P_{\mu}(\mathrm{Ad}\,f^{\mu}(\mu_{i}))\in \cH^{0,1}(X,\mathrm{End}\, E_{\rho^{\mu}}),\quad i=1,\dots,d.$$

In complex coordinates we have the following expression for the Narasimhan-Atiyah-Bott symplectic form, 
\begin{equation}\label{anb-complex}
\omega_{NAB}\left(\frac{\del}{\del\vep_{\mu}},\frac{\del}{\del\bar\vep_{\nu}}\right)=\frac{\sqrt{-1}}{2}\la \mu,\nu\ra,
\end{equation}
where $\dfrac{\del}{\del\vep_{\mu}}$ and $\dfrac{\del}{\del\bar\vep_{\nu}}$ are the holomorphic and antiholomorphic tangent vectors at $\{E\}\in\cN$ corresponding to $\mu,\nu\in\cH^{0,1}(X,\mathrm{End}\, E)$ respectively.

\begin{remark}\label{hol-reps}
Note that unitary representations $\rho^{\mu}$  do not depend holomorphically on $\mu=\vep_{1}\mu_{1}+\cdots+\vep_{d}\mu_{d}$. However, representations $\rho_{\mu}$, constructed in the proof of Proposition \ref{B-coordinates}, depend holomorphically on complex coordinates $(\vep_{1},\dots,\vep_{d})$ in $U$. 
\end{remark}

\subsection{Eichler integrals}\label{Eichler} Here we briefly review Eichler integrals of weight $0$, referring to \cite{Meneses} for the general theory and further references. 
Namely, for harmonic $\mu\in \cH^{0,1}(X,\mathrm{End}\, E)$ let $\cE$ be the corresponding Eichler integral --- the anti-holomorphic matrix-valued function on $\HH$, defined by
$$\cE(z)=\int_{z_{0}}^{z}\mu(\zeta)d\bar\zeta.$$
It follows from \eqref{mu-transform} that it satisfies
$$\cE(\gamma z)=\mathrm{Ad}\,\rho(\gamma)\cE(z)+\vp(\gamma),$$
where $n\times n$ complex matrices $\vp(\gamma)$ are the Eichler-Shimura periods of $\mu$,
$$\vp(\gamma)=\int_{z_{0}}^{\gamma z_{0}}\mu(z)d\bar{z}, \quad\gamma\in\Gamma.$$
The Eichler-Shimura periods satisfy
\begin{equation}\label{E-S}
\vp(\gamma_{1}\gamma_{2})=\vp(\gamma_{1})+ \mathrm{Ad}\,\rho(\gamma_{1})\vp(\gamma_{2}),\quad\text{for all}\quad\gamma_{1}, \gamma_{2}\in\Gamma,
\end{equation}
so $\vp\in Z^{1}(\Gamma,\frak{g}_{\mathrm{Ad}\,\rho})$, the space of $1$-cocycles for the group $\Gamma$ with values in the Lie algebra $\frak{g}=\frak{gl}(n,\CC)$ --- a left $\Ga$-module with respect to the adjoint action $\mathrm{Ad}\,\rho$  (see \S\,\ref{goldman}). We have the period map
$$\cH^{0,1}(X,\mathrm{End}\, E)\ni\mu \mapsto\cP(\mu)=[\vp]\in H^{1}(\Gamma,\frak{g}_{\mathrm{Ad}\,\rho}),$$
where $[\vp]$ is the cohomology class of a cocycle $\vp$. 

Similarly, for harmonic $\theta\in \cH^{1,0}(X,\mathrm{End}\, E)$ the Eichler integral $\Theta$ and Eichler-Shimura periods $\psi(\g)$ are defined as
$$\Theta(z)=\int_{z_{0}}^{z}\theta(\zeta)d\zeta\quad\text{and}\quad\Theta(\gamma z)=\mathrm{Ad}\,\rho(\gamma)\Theta(z)+ \psi(\g).$$
If $\theta=\mu^{*}$, then $\psi(\g)=\vp(\g)^{*}$.

The following result is well-known (see \cite{NS1,Gunning} and especially \cite{Meneses}). For convenience of the reader we present a simple proof.

\begin{lemma} \label{iso-period} The period map $\cP$ for a stable bundle $E$ establishes the isomorphism
$$\cH^{1,0}(X,\mathrm{End}\, E)\oplus\cH^{0,1}(X,\mathrm{End}\, E)\simeq H^{1}(\Gamma,\frak{g}_{\mathrm{Ad}\,\rho}).$$
Moreover, if for all $\gamma\in\Gamma$ the Eichler-Shimura periods $\vp(\gamma)\in\frak{u}(n)$ --- the Lie algebra of $\mathrm{U}(n)$ --- then $\cE=0$.
\end{lemma}
\begin{proof} If $\vp(\g)=\mathrm{Ad}\,\rho(\g)v-v$ for some $v\in\frak{g}$, then $\cE(z)+v$ is $\mathrm{Ad}\,\rho$-invariant Eichler integral. Since the representation $\rho$ is unitary, 
$\cE(z)+v$ is constant and $\mu=0$. If $\cP(\mu)=\cP(\theta)$, then $\cE-\Theta$ is harmonic 
$\mathrm{Ad}\,\rho$-invariant matrix-valued function on $\HH$. Since the representation $\rho$ is unitary, $\cE-\Theta$ is constant (see \cite[Prop.~4.2]{NS1}) and $\mu=\theta=0$.
The complex vector spaces $\cH^{1,0}(X,\mathrm{End}\, E)\oplus\cH^{0,1}(X,\mathrm{End}\, E)$ and $H^{1}(\Gamma,\frak{g}_{\mathrm{Ad}\,\rho})$ for a stable bundle $E$ have the
same dimension, so $\cP$ is an isomorphism. 

If $\vp(\gamma)\in\frak{u}(n)$, then $\cE+\cE^{*}$ is harmonic 
$\mathrm{Ad}\,\rho$-invariant matrix-valued function on $\HH$, and as above $\mu=0$. 
\end{proof}

\subsection{Infinitesimal deformations}\label{2.5}
For $\mu\in \cH^{0,1}(X,\mathrm{End}\, E)$ and for small enough $t\in\RR$ and $z\in\HH$ put 
$$\dot{f}^{\mu}(z)=\left.\frac{d}{dt}\right|_{t=0}f^{t\mu}(z), \quad\dot{f}_{\mu}(z)=\left.\frac{d}{dt}\right|_{t=0}f_{t\mu}(z),\quad\text{and}\quad \dot{g}^{\mu}(z)=\left.\frac{d}{dt}\right|_{t=0}g^{t\mu}(z).$$

The following simple result is a vector bundle analog of the classical Ahlfors' result --- formula (1.21) in \cite{Ahlfors} --- for the variation of a family of quasi-conformal mappings with harmonic Beltrami differential.
\begin{lemma} \label{a-lemma} For $\mu\in \cH^{0,1}(X,\mathrm{End}\, E)$ we have
$$\dot{f}^{\mu}(z)=\cE(z)-\cE(z)^{\ast}\quad\text{and}\quad \dot{g}^{\mu}(z)=-\cE(z)^{\ast},$$
where $\cE$ is the Eichler integral of $\mu$.
\end{lemma}
\begin{proof} Replace in the equation \eqref{dbar} $\mu$ by $t\mu$ and differentiate it  with respect to $t$ at $t=0$. Using that $\left.f^{t\mu}(z)\right|_{t=0}=\left.f_{t\mu}(z)\right|_{t=0}=I$ we obtain
$$\frac{\del}{\del\bar{z}}\dot{f}^{\mu}(z)=\frac{d}{d\bar{z}}\dot{f}_{\mu}(z)=\mu(z),$$
and since $\dot{f}_{\mu}(z_{0})=0$ we have $\dot{f}_{\mu}=\cE$ and
$$\dot{f}^{\mu}=\cE+F,$$
where the function $F(z)$ is holomorphic. Since $\rho^{t\mu}(\gamma)\in \mathrm{U}(n)$, we have $\dot{\rho}^{\mu}(\gamma)\rho^{-1}(\gamma)\in\frak{u}(n)$, where
$$\dot{\rho}^{\mu}(\gamma)=\left.\frac{d}{dt}\right|_{t=0}\rho^{t\mu}(\g).$$
It follows from the transformation law in part (i) of Proposition \ref{B-coordinates} that
\begin{equation}\label{transform-law}
\dot{f}^{\mu}(\gamma z)=\mathrm{Ad}\,\rho(\gamma)\dot{f}^{\mu}(z)+\dot{\rho}^{\mu}(\gamma)\rho^{-1}(\gamma),
\end{equation}
so as in the proof of Lemma \ref{iso-period}, harmonic $\mathrm{Ad}\,\rho$-invariant function $(\cE+F)+(\cE+F)^{*}$ is $cI$. Since $\det f^{t\mu}(z_{0})=1$, we have $\tr\dot{f}^{\mu}(z_{0})=0$, so $c=0$ and   
$F+\cE^{\ast}=0$, which proves the first formula.  Since $\dot{g}^{\mu}=\dot{f}^{\mu}-\dot{f}_{\mu}$, we get the second formula.
\end{proof}

\begin{corollary}\label{cor-0} If $\dot{\rho}^{\mu}\rho^{-1}\!\in B^{1}(\Gamma,\frak{g}_{\mathrm{Ad}\,\rho})$, i.e., $\dot{\rho}^{\mu}(\gamma)\rho^{-1}(\gamma)=\mathrm{Ad}\,\rho(\g)v-v$ for some $v\in\frak{g}$ and all $\g\in\Gamma$, then $\mu=0$.
\end{corollary}
\begin{proof} It follows from \eqref{transform-law} that in this case $\dot{f}^{\mu}+v$ is $\mathrm{Ad}\,\rho$-invariant harmonic matrix-valued function on $\HH$, and since the representation $\rho$ is unitary, it is a constant function.
\end{proof}
Introducing for small $\vep\in\CC$ notation
$$\dot{f}_{+}^{\mu}=\left.\frac{\del}{\del\vep}\right|_{\vep=0}f^{\vep\mu}\quad\text{and}\quad \dot{f}_{-}^{\mu}=\left.\frac{\del}{\del\bar\vep}\right|_{\vep=0}f^{\vep\mu},$$
we get
$$\dot{f}_{+}^{\mu}=\frac{1}{2}(\dot{f}^{\mu}-i\dot{f}^{i\mu})\quad\text{and}\quad \dot{f}_{-}^{\mu}=\frac{1}{2}(\dot{f}^{\mu}+i\dot{f}^{i\mu})$$ 
(here $i=\sqrt{-1}$), and from Lemma \ref{a-lemma} we obtain
\begin{corollary} \label{cor-1}
\begin{align*}
\dot{f}_{+}^{\mu}=\cE,\quad \dot{f}_{-}^{\mu}=-\cE^{*}.
\end{align*}
\end{corollary}

\section{Goldman symplectic form} \label{goldman}

\subsection{Character variety} \label{ch-variety} Let $G$ be 
a complex (or real) 
Lie group that preserves a non-degenerate symmetric bilinear form $B$ on its Lie algebra $\frak{g}$, and let $\cK_{G}$ be the corresponding $G$-character variety, a complex (or real) manifold of complex (or real) dimension $(2g-2)\dim\frak{g}+2$. Here  we only consider the group $G=\mathrm{GL}(n,\CC)$ and its compact real form $G_{\RR}=\mathrm{U}(n)$, so
$$\cK_{G}=\mathrm{Hom}_{0}(\pi_{1},G)/G,$$
the subscript $0$ stands for irreducible representations, and use $B(u,v)=\tr uv$ on $\frak{g}\simeq\mathrm{End}\,\CC^{n}$. The character variety $\cK=\cK_{G_{\RR}}$ was defined in \S \ref{m space}.

It is well-known \cite{Goldman1} that the holomorphic tangent space $T_{[\sigma]}\cK_{G}$ at $[\sigma]$ is naturally identified with the cohomology group 
$$H^{1}(\Gamma,\frak{g}_{\mathrm{Ad}\,\si})=Z^{1}(\Gamma,\frak{g}_{\mathrm{Ad}\,\si})/B^{1}(\Gamma,\frak{g}_{\mathrm{Ad}\,\si}),$$
where we use $\pi_{1}\simeq\Gamma$. Here $\frak{g}$ is understood as a left $\Ga$-module with respect to the action $\mathrm{Ad}\,\si$, and a $1$-cocycle $\chi\in Z^{1}(\Gamma,\frak{g}_{\mathrm{Ad}\,\si})$ is a map $\chi: \Gamma\rightarrow \frak{g}$ satisfying
\begin{equation} \label{par-1}
\chi(\gamma_{1}\gamma_{2})=\chi(\gamma_{1})+\si(\gamma_{1})\cdot\chi(\gamma_{2}), \quad\gamma_{1},\gamma_{2}\in \Gamma.
\end{equation}
Here and in what follows we denote by the dot the adjoint action of $G$ on $\frak{g}$.

Denote by $[\chi]$ the cohomology class of a $1$-cocycle $\chi$.  
The Goldman symplectic form $\omega_{G}$ is a holomorphic $(2,0)$-form (or real $(1,1)$-form) on the character variety $\cK_{G}$, defined by
\begin{equation}\label{G}
\omega_{G}([\chi_{1}],[\chi_{2}])=\la [\chi_{1}]\cup[\chi_{2}]\ra([X]),\quad\text{where}\quad[\chi_{1}],[\chi_{2}]\in T_{[\si]}\cK_{G}.
\end{equation}
Here $[X]$ is the fundamental class of $X$ under the isomorphism $H_{2}(X,\ZZ)\simeq H_{2}(\Gamma,\ZZ)$, and $\la [\chi_{1}]\cup[\chi_{2}]\ra\in H^{2}(\Gamma,\RR)$ is a composition of the cup product in the cohomology and of the invariant bilinear form. At the cocycle level it is given explicitly by
\begin{align*}
\la \chi_{1}\cup\chi_{2}\ra(\gamma_{1},\gamma_{2})&=B(\chi_{1}(\gamma_{1}),\si(\gamma_{1})\cdot\chi(\gamma_{2}))\\
&=-B(\chi_{1}(\gamma^{-1}_{1}),\chi(\gamma_{2})),\quad\gamma_{1},\gamma_{2}\in\Gamma.
\end{align*}
The right-hand side in \eqref{G} does not depend on a choice of representatives $\chi_{1}, \chi_{2} \in Z^{1}(\Gamma,\frak{g}_{\mathrm{Ad}\,\si})$ of the cohomology classes $[\chi_{1}], [\chi_{2}]\in H^{1}(\Gamma,\frak{g}_{\mathrm{Ad}\,\si}) $, and we will use the notation $\omega_{G}(\chi_{1},\chi_{2})$.

Let $a_{k},b_{k}$, $k=1,\dots,g$, be standard generators of the group $\Gamma$, satisfying the single relation
$$R_{g}=\prod_{k=1}^{g}[a_{k},b_{k}]=1,\quad\text{where}\quad [a,b]=aba^{-1}b^{-1}.$$
The fundamental class $[X]$ can be realized as the following $2$-cycle in the group homology (see \cite{Goldman1, Tak} and references therein) 
\begin{equation} \label{2-cycle}
c=\sum_{k=1}^{g}\left\{\left(\frac{\del R}{\del a_{k}}, a_{k}\right) +\left(\frac{\del R}{\del b_{k}}, b_{k}\right)\right\}\in H_{2}(\Gamma,\ZZ),
\end{equation}
where $R=R_{g}$ and
\begin{equation*}
R_{k} =  \prod_{i=1}^{k}[a_{i},b_{i}],\quad k=1,\dots, g.
\end{equation*}
Here
\begin{equation*} 
\frac{\del R}{\del a_{k}}= R_{k-1}-R_{k}b_{k}\quad\text{and}\quad \frac{\del R}{\del b_{k}}= R_{k-1}a_{k}-R_{k},
\end{equation*}
where derivatives are understood in the sense of Fox free differential calculus and the relation $R=1$ is set after the differentiation.
In these notations \eqref{G} takes the following form
$$\omega_{\mathrm{G}}(\chi_{1},\chi_{2})=-\sum_{k=1}^{g}B\left(\chi_{1}\!\!\left(\# \frac{\del R}{\del a_{k}}\right),\chi_{2}(a_{k})\right) + B\left(\chi_{1}\!\!\left(\# \frac{\del R}{\del b_{k}}\right),\chi_{2}(b_{k})\right).$$
Here $\chi$ is extended to a linear map defined on the integral group ring $\ZZ[\Gamma]$, and $\#$ denotes the natural anti-involution on $\ZZ[\Gamma]$, 
$$\#\left(\sum n_{j}\gamma_{j}\right)=\sum n_{j}\gamma_{j}^{-1}.$$

It is convenient to use the dual generators of $\Gamma$,
\begin{equation}\label{dual-gen}
\alpha_{k}=R_{k-1}b^{-1}_{k}R_{k}^{-1},\;\beta_{k}=R_{k}a_{k}^{-1}R_{k-1}^{-1},\quad k=1,\dots,g.
\end{equation} 
They satisfy
$$[\alpha_{k},\beta_{k}]=R_{k-1}R_{k}^{-1},$$
so that
\begin{equation}\label{R-dual}
\mathcal{R}_{k}=\prod_{i=1}^{k}[\alpha_{i},\beta_{i}]=R_{k}^{-1},\quad \mathcal{R}_{g}=1
\end{equation}
and
\begin{equation} \label{gen-dual}
a_{k}^{-1}=\R_{k}\beta_{k}\R_{k-1}^{-1},\quad b_{k}^{-1}=\R_{k-1}\alpha_{k}\R^{-1}_{k}.
\end{equation}
We have\footnote{Correcting two obvious typos, missing negative signs, in formulas in \cite[Remark 3]{Tak}.}
$$\# \frac{\del R}{\del a_{k}}=R_{k-1}^{-1}-R_{k-1}^{-1}\alpha_{k}\quad\text{and}\quad \# \frac{\del R}{\del b_{k}}= R_{k}^{-1}\beta_{k}-R_{k}^{-1},$$
so the Goldman form can be conveniently written as
\begin{gather}
\omega_{\mathrm{G}}(\chi_{1},\chi_{2})=\nonumber\\
\sum_{k=1}^{g}B(\chi_{1}(\alpha_{k}),\si(R_{k-1})\!\cdot\!\chi_{2}(a_{k})) - B(\chi_{1}(\beta_{k}),\si(R_{k})\!\cdot\!\chi_{2}(b_{k})).\label{dual}
\end{gather}

\subsection{Goldman theorem}  \label{3-2}
Fuchsian uniformization of genus $g>1$ 
Riemann surfaces determines a real-analytic map of the Teichm\"{u}ller space $T_{g}$ to the $\mathrm{PSL}(2,\RR)$-character variety. 
Proposition 2.5 in \cite{Goldman1} asserts that the the pullback of Goldman symplectic form by this map is the Weil-Petersson symplectic form on $T_{g}$. As W. Goldman notes in \cite{Goldman1}, it is a reformulation, in a more invariant form using modern notation, of classic results of M. Eichler and G. Shimura on the periods of automorphic forms (see \S\,8.2 in \cite{Shimura} and references therein).

Likewise, the Narasimhan-Seshadri  theorem determines a real-analytic map $\imath$ of the moduli space $\cN$ to the $\mathrm{U}(n)$-character variety $\cK$ by assigning to each stable bundle
corresponding unitary representation:  
$$\cN\ni \{E\}\mapsto\imath(E)=[\rho]\in\cK.$$
Using the $\RR$-linear isomorphism between real and holomorphic tangent spaces,  one can think of the differential of $\imath$ as the map

$$T_{E}\cN\ni \mu\mapsto \imath_{\ast}(\mu)=\chi\in T_{\rho}\cK,$$
and according to Lemma \ref{a-lemma},
$$\chi(\gamma)=\vp(\gamma)-\vp(\gamma)^{*}\in\frak{u}(n),$$
where $\vp(\gamma)$ are the Eichler-Shimura periods of a harmonic $(0,1)$-form $\mu$ (see \S\ref{Eichler}).
 
Theorem \ref{E-S-G} below is a vector bundle analogue of Proposition 2.5 in \cite{Goldman1} and is attributed to Goldman.
Likewise, it is basically a reformulation,  using a more invariant form and modern notation, of the Riemann bilinear relations for the bundle-valued holomorphic differentials
on a Riemann surface, as in Theorem 21 in R. Gunning's lectures \cite{Gunning}. 

We summarize these relations in Theorem \ref{th: E-S} below. Its proof is based on a simple computation, which we present here for convenience of the reader. It uses a detailed structure of a fundamental domain $F$ associated with the standard generators $a_{k},b_{k}$ of $\Gamma$ (see \cite{Hejhal} and references therein). Succinctly, the fundamental domain $F$ we use is an oriented topological $4g$-gon
with the base point $z_{0}\in \HH$, whose ordered vertices are consecutive quadruples
$$(R_{k}z_{0},R_{k}a_{k+1}z_{0}, R_{k}a_{k+1}b_{k+1}z_{0}, R_{k}a_{k+1}b_{k+1}a^{-1}_{k+1}z_{0}),\quad k=0,\dots, g-1.$$
Corresponding $A$ and $B$ edges of $F$ are analytic arcs 
$$A_{k}=(R_{k-1}z_{0}, R_{k-1}a_{k}z_{0})\quad \text{and}\quad B_{k}=(R_{k}z_{0},R_{k}b_{k}z_{0}),$$
and corresponding dual edges are 
$$A_{k}'= (R_{k}b_{k}z_{0},R_{k}b_{k}a_{k}z_{0})\quad \text{and}\quad B_{k}' =(R_{k-1}a_{k}z_{0},R_{k}b_{k}a_{k}z_{0}),$$ 
$k=1,\dots,g$. (see Fig. 1 for a typical fundamental domain for a
group $\Gamma$ with $g=2$). 
We have
\begin{equation} \label{f-domain}
\del F=\sum_{i=1}^{2g}(S_{i}-\lambda_{i}(S_{i})),
\end{equation}
where 
\begin{equation}\label{S-lambda}
S_{k}=A_{k},\;\;S_{k+g}=-B_{k}\quad \text{and}\quad \lambda_{k}=\alpha^{-1}_{k},\;\; \lambda_{k+g}=\beta^{-1}_{k},
\end{equation}
$k=1,\dots,g$.
\vspace{5mm}

\begin{tikzpicture}
\tikzstyle{every node}=[font=\small]
\draw (-6,0) -- (6,0) coordinate[];
\fill[gray!10]   (2,2) to (3,4) to (2,5) to (0,5) to (-2,4) to (-3,3) to (-2,2) to (0,1);

\begin{scope}[thick, decoration={markings, mark=at position 0.5 with {\arrow{>}}}]         
\draw[fill=white!300, line width=0.8pt, postaction ={decorate}]
(2,2) to [bend left] (3,4);
\draw[fill=white!300,,line width=0.8pt, postaction ={decorate}] 
(3,4) to [bend left] (2,5);
\draw[fill=white!300,,line width=0.8pt, postaction ={decorate}] 
(0,5) to [bend right] (2,5);
\draw[fill=white!300,,line width=0.8pt, postaction ={decorate}] 
(-2,4) to [bend right] (0,5);
\draw[fill=white!300,,line width=0.8pt,postaction ={decorate}] 
(-2,4) to [bend left] (-3,3);
\draw[fill=white!300, line width=0.8pt,postaction ={decorate}] 
(2,2) to [bend right] (0,1);
\draw[fill=white!300, line width=0.8pt,postaction ={decorate}] 
(0,1) to [bend right] (-2,2);
\draw[fill=white!300,line width=0.8pt, postaction ={decorate}] 
(-3,3) to [bend left] (-2,2);
  \end{scope}
\node[right] at (2,1.95) {$z_{0}$};
\node[below] at (1,1.6) {$B_{2}$};
\node[right] at (3,4) {$a_{1}z_{0}$};
\node[below] at (2.6,3.4) {$A_{1}$};
\node at (0,3) {$F$};
\node[above] at (2.2,5) {$a_{1}b_{1}z_{0}$};
\node[above] at (2.5,4.35) {$B'_{1}$};
\node[above] at (-.2,5) {$a_{1}b_{1}a_{1}^{-1}z_{0}$};
\node[above] at (1,4.7) {$A'_{1}$};
\node[above] at (-2,4) {$R_{1}z_{0}$};
\node[above] at (-1,4.2) {$B_{1}$};
\node[left] at (-3,3) {$R_{1}a_{2}z_{0}$};
\node[left] at (-2.25, 3.6) {$A_{2}$};
\node[below] at (-2,2) {$R_{1}a_{2}b_{2}z_{0}$};
\node[below] at (-2.6,2.7) {$B'_{2}$};
\node[below] at (-1,1.6) {$A'_{2}$};
\node[below] at (0,1) {$R_{1}a_{2}b_{2}a_{2}^{-1}z_{0}$};
\node at (0,-1) {Figure 1};
\filldraw [black] (-3,3) circle [radius=1pt];
\filldraw [black] (0,1) circle [radius=1pt];
\filldraw [black] (2,2) circle [radius=1pt];
\filldraw [black] (3,4) circle [radius=1pt];
\filldraw [black] (2,5) circle [radius=1pt];
\filldraw [black] (0,5) circle [radius=1pt];
\filldraw [black] (-2,4) circle [radius=1pt];
\filldraw [black] (-2,2) circle [radius=1pt];
\end{tikzpicture}

\noindent

\begin{theorem} \label{th: E-S} Let $\mu_{1}, \mu_{2}\in\cH^{0,1}(X,\mathrm{End}\, E_{\rho})$, where $\rho:\Gamma\to\mathrm{U}(n)$, and let $\vp_{1}$ and $\vp_{2}$ be their corresponding Eichler-Shimura periods. We have the following analog of Riemann bilinear relations:
$$\omega_{G}(\vp_{1},\vp_{2}^{*}) = \sqrt{-1}\,\la \mu_{1},\mu_{2}\ra 
\quad\text{and}\quad \omega_{G}(\vp_{1},\vp_{2}) =0.$$
\end{theorem}
\begin{proof} By Stokes theorem,
\begin{gather*}
\la\mu_{1},\mu_{2}\ra =\sqrt{-1}\iint_{F}\tr\left(\mu_{1}(z)\mu_{2}(z)^{*}\right) dz\wedge d\bar{z}
 =-\sqrt{-1}\iint_{F}d\left\{\tr(\cE_{1}(z)\mu_{2}(z)^{*})dz\right\}\\
=-\sqrt{-1}\int_{\del F}\tr\left(\cE_{1}(z)\mu_{2}(z)^{*}\right)dz.
\end{gather*}
Using \eqref{f-domain}, we get
\begin{gather*}
\int_{\del F}\tr(\cE_{1}(z)\mu_{2}(z)^{*})dz=\sum_{i=1}^{2g}\left\{\int_{S_{i}}\tr\left(\cE_{1}(z)\mu_{2}(z)^{*}\right)dz -\int_{\lambda_{i}(S_{i})}\tr\left(\cE_{1}(z)\mu_{2}(z)^{*}\right)dz\right\}\\
 =-\sum_{i=1}^{2g}\int_{S_{i}}\tr(\vp_{1}(\lambda_{i}) \rho(\lambda_{i})\cdot\mu_{2}(z)^{*})dz
  =\sum_{i=1}^{2g}\tr\left(\vp_{1}(\lambda^{-1}_{i})\int_{S_{i}}\mu_{2}(z)^{*}dz\right).
\end{gather*}
Using \eqref{S-lambda}, we obtain
$$\int_{S_{k}}\mu_{2}(z)^{*}dz=\int_{A_{k}}\mu_{2}(z)^{*}dz=\int_{R_{k-1}z_{0}}^{R_{k-1}a_{k}z_{0}}\mu_{2}(z)^{*}dz=\rho(R_{k-1})\cdot\vp_{2}(a_{k})^{*}$$
and
$$\int_{S_{k+g}}\mu_{2}(z)^{*}dz=-\int_{B_{k}}\mu_{2}(z)^{*}dz=-\int_{R_{k}z_{0}}^{R_{k}b_{k}z_{0}}\mu_{2}(z)^{*}dz=-\rho(R_{k})\cdot\vp_{2}(b_{k})^{*}.$$
Therefore,
\begin{gather*}
\la\mu_{1},\mu_{2}\ra= -\sqrt{-1}\sum_{k=1}^{g}\left\{\tr(\vp_{1}(\alpha_{k})\rho(R_{k-1})\cdot\vp_{2}(a_{k})^{*})-\tr(\vp_{1}(\beta_{k})\rho(R_{k})\cdot\vp_{2}(b_{k})^{*})\right\}\\
=-\sqrt{-1}\,\omega_{G}(\vp_{1},\vp_{2}^{*}).
\end{gather*}
This proves the first bilinear relation. The second one follows immediately, since 
$$\int_{X}\mu_{1}\wedge \mu_{2}=0.\qedhere$$
\end{proof}
\begin{remark}\label{one-zero} Equivalently, Theorem \ref{th: E-S} can be stated in terms of the harmonic $(1,0)$-forms. Namely, let $\theta_{1}, \theta_{2}\in\cH^{1,0}(X,\mathrm{End}\, E_{\rho})$, where $\rho:\Gamma\to\mathrm{U}(n)$, and let $\psi_{1}$ and $\psi_{2}$ be their corresponding Eichler-Shimura periods. Then
$$\omega_{G}(\psi_{1},\psi_{2}^{*}) = \int_{X}\theta_{1}\wedge \theta_{2}^{*}\quad\text{and}\quad \omega_{G}(\psi_{1},\psi_{2}) =0.$$
According to Lemma \ref{iso-period}, $\cH^{1,0}(X,\mathrm{End}\, E_{\rho})$ and $\cH^{0,1}(X,\mathrm{End}\, E_{\rho})$ are complementary Lagrangian subspaces in the complex symplectic vector space
$H^{1}(\Gamma,\frak{g}_{\mathrm{Ad}\,\rho})$. 
\end{remark}
As immediate corollary of Theorem \ref{th: E-S}, we have the following vector bundle analog of the Eichler-Shimura-Goldman result --- Proposition 2.5 in \cite{Goldman1}.
\begin{theorem}\label{E-S-G} Under the map $\imath:\cN\to\cK$ of the moduli space $\cN$ to the $\mathrm{U}(n)$-character variety $\cK$,
$$\imath^{*}\omega_{G}=-4\, \omega_{NAB}.$$
\end{theorem}
\begin{proof}
Since $\imath_{*}(\mu)=\vp-\vp^{*}$, using Theorem \ref{th: E-S} we obtain
\begin{align*}
\omega_{G}(\imath_{*}(\mu_{1}),\imath_{*}(\mu_{2})) & = \omega_{G}(\vp_{1}-\vp_{1}^{*},\vp_{2}-\vp_{2}^{*})\\
&=-\sqrt{-1}(\la\mu_{1},\mu_{2}\ra - \la\mu_{2},\mu_{1}\ra)\\
&=-4\,\omega_{NAB}(v_{1},v_{2}), 
\end{align*}
where $v_{1}$ and $v_{2}$ are real tangent vectors corresponding to the holomorphic tangent vectors $\mu_{1}$ and $\mu_{2}$.
\end{proof}
\begin{remark} 
Here we are using $\tr$ as the invariant bilinear form $B$ on Lie algebras $\frak{gl}(n,\CC)$ and $\frak{u}(n)$. The form $B$ is negative-definite on $\frak{u}(n)$, which explains the negative sign in Theorem \ref{E-S-G}.
\end{remark}
\section{Zero curvature connections}  \label{zero-curv} Denote by $\cA(E)$ the zero curvature connections in $E$, compatible with the holomorphic structure in
$E$. It is an affine space over the vector space  $\cH^{1,0}(X,\mathrm{End} \,E)$, and the spaces $\cA(E)$ for $\{E\}\in\cN$ combine into a holomorphic affine
bundle $\cA\to\cN$ over the holomorphic cotangent bundle $T^{*}\cN\to\cN$.  

\subsection{The Riemann-Hilbert correspondence} \label{RH} Let $\cK_{\CC}$ be 
the $ \mathrm{GL}(n,\CC)$-character variety. 
The Riemann-Hilbert correspondence is the monodromy map 
$$\cQ: \cA\to\cK_{\CC},$$
and is described
as follows. Let $\nabla=d+A\in\cA(E)$ be a zero curvature connection. Its holonomy is a representation $\sigma$ of $\pi_{1}\simeq\Gamma$ in $\mathrm{GL}(n,\CC)$. Since  $E$ is stable, $\sigma$ is irreducible, and we put $\cQ(\nabla)=\sigma\in\cK$. 

Concretely, using the Narasimhan-Seshadri theorem, we realize the bundle $E$ as a quotient bundle $E_{\rho}$, so that
a zero curvature connection is $\nabla=d+A$, where  $A(z)$ is a holomorphic matrix-valued $\Gamma$-automorphic form $A(z)$ of weight $2$ with representation $\mathrm{Ad}\,\rho$  on $\HH$, 
$$A(\gamma z)\gamma'(z)=\mathrm{Ad}\,\rho(\gamma) A(z)\quad \text{for all}\quad \gamma\in\Gamma.$$
The parallel transport equation
$$\frac{dg}{dz}+Ag=0$$
has a unique holomorphic solution: a non-degenerate matrix-valued  function $g(z)$ on $\HH$, satisfying $g(z_{0})=I$ and 
\begin{equation}\label{g-iso}
g(\gamma z)=\rho(\gamma)g(z)\sigma(\gamma)^{-1}\quad \text{for all}\quad \gamma\in\Gamma,
\end{equation}
thus specifying a representation $\sigma:\Gamma\to \mathrm{GL}(n,\CC)$.
The map $g$ is an isomorphism $E_{\si}\simeq E_{\rho}$ of the quotient bundles, so
the representation $\sigma$ is irreducible and determines constant transition functions for the bundle $E$. 

Conversely, a choice of constant transitions functions for the bundle $E$ determines a representation $\rho$, and the image of the fiber $\cA(E)$ under the Riemann-Hilbert correspondence is the set of equivalence classes of constant transition functions of the stable
vector bundle $E$. 

\subsection{Holomorphic representations} \label{holo-rep} Generalizing Remark \ref{hol-reps} in \S\,\ref{c-coordinates}, here we define, for each coordinate chart, a holomorphic family of irreducible representations. Namely, choose $\{E\}\in\cN$ and $\nabla=d+A\in\cA(E)$ and, using the map $g$ above, realize the bundle $E=E_{\rho}$
as $E_{\si}$. The quotient bundle $E_{\si}$ is a local system with de Rham differential, so Dolbeault cohomology group $H^{0,1}(X,\mathrm{End}\,E_{\si})$ can be realized 
as $H^{0,1}_{\mathrm{dR}}(X,\mathrm{End}\,E_{\si})$ ---the space of antiholomorphic matrix-valued functions $\mu(z)$ on $\HH$,
satisfying
\begin{equation*}
\mu(\gamma z)\overline{\g'(z)}=\mathrm{Ad}\,\si(\gamma)\mu(z)\quad \text{for all}\quad \gamma\in\Gamma.
\end{equation*}
The solution $f(z)$ of the ordinary differential equation 
\begin{equation}\label{diff-eq}
\frac{df}{d\bar{z}}(z)=f(z)\mu(z)
\end{equation}
normalized by $f(z_{0})=I$ satisfies 
\begin{equation}\label{f-ordinary-transform}
f(\g z)=\sigma_{\mu}(\g)f(z)\sigma(\g)^{-1},
\end{equation}
and for small enough $\mu$ determines a family $\si_{\mu}$ of irreducible representations. Choose a basis $\mu_{i}$ in $H^{0,1}_{\mathrm{dR}}(X,\mathrm{End}\,E_{\si})$ and define complex coordinates in a neighborhood of $E_{\si}$ by $\mu=\vep_{1}\mu_{1}+\cdots+\vep_{d}\mu_{d}$. These coordinates are the same as in Proposition \ref{B-coordinates}, and we obtain a holomorphic
family of representations parameterized by $\si$; construction in Remark \ref{hol-reps} corresponds to the case $A=0$.

Indeed, it is sufficient to verify that the map 
$$H^{0,1}_{\mathrm{dR}}(X,\mathrm{End}\,E_{\si})\ni\mu\to \tilde\mu=P(\hat\mu)\in \cH^{0,1}(X,\mathrm{End}\,E_{\rho}),\quad\hat\mu=\mathrm{Ad}\,g(\mu),$$
is injective, and hence is an isomorphism. Namely, if $\tilde\mu=0$, then we have for all $\theta\in \cH^{1,0}(X,\mathrm{End}\,E_{\rho})$,
$$0=\int_{X} \tilde\mu\wedge \theta=\int_{X} \hat\mu\wedge \theta=\int_{X} \mu\wedge \mathrm{Ad}\,g^{-1}(\theta).$$
However, since $g: E_{\si}\to E_{\rho}$ is an isomorphism of holomorphic vector bundles, the map 
$$\mathrm{Ad}\,g^{-1}: \cH^{1,0}(X,\mathrm{End}\,E_{\rho})\to H^{1,0}_{\mathrm{dR}}(X,\mathrm{End}\,E_{\si})$$ 
is an isomorphism, so $\mu=0$. 

The inner product in $H^{0,1}_{\mathrm{dR}}(X,\mathrm{End}\,E_{\si})$ is determined by requiring that the isomorphism $H^{0,1}_{\mathrm{dR}}(X,\mathrm{End}\,E_{\si})\simeq \cH^{0,1}(X,\mathrm{End}\,E_{\rho})$ is an isometry.

This discussion can be summarized by the following commutative diagram

\begin{equation}\label{cd-1}
\begin{tikzcd}
 E_{\rho} \arrow{r}{f^{\hat\mu}} & E_{\rho^{\hat\mu}}\\
    E_{\si}   \arrow{u}[swap]{g}\arrow{r}{f_{\mu}} & E_{\si_{\mu}}  \arrow{u}[swap]{g^{\mu}}
\end{tikzcd}.
\end{equation}
Here $f_{\mu}$ satisfies ordinary differential equation \eqref{diff-eq} with $\mu\in H^{0,1}_{\mathrm{dR}}(X,\mathrm{End}\,E_{\si})$ and the property \eqref{f-ordinary-transform}, while the map $f^{\hat\mu}$ satisfies differential equation \eqref{dbar} with $\hat\mu\in\Omega^{0,1}(X,\mathrm{End}\,E_{\rho})$ 
and the property
$$f^{\hat\mu}(\g z)=\rho^{\hat\mu}(\g)f^{\hat\mu}(z)\rho(\g)^{-1},$$
where $\rho^{\hat\mu}$ is an irreducible unitary representation such that $E_{\si_{\mu}}\simeq E_{\rho^{\hat\mu}}$. If $\mu=\vep\mu_{1}+\cdots+\vep_{d}\mu_{d}$, where $\mu_{i}$ is a basis in $H^{0,1}_{\mathrm{dR}}(X,\mathrm{End}\,E_{\si})$, then corresponding vector fields $\dfrac{\del}{\del\vep_{i}}$ at $E_{\si_{\mu}}$ are given  by $\mathrm{Ad}\,f_{\mu}(\mu_{i})\in H_{\mathrm{dR}}^{0,1}(X,\mathrm{End}\, E_{\si_{\mu}})$, while at $E_{\rho^{\hat\mu}}$ they are given by $P(\mathrm{Ad}\,f^{\hat\mu}({\hat\mu}_{i}))\in \cH^{0,1}(X,\mathrm{End}\, E_{\rho^{\hat\mu}})$.

\subsection{Holomorphic sections} \label{hol sections} The Narasimhan-Seshadri connections $\nabla_{E}=d+A_{E}$ (see \S\,\ref{1.1})
define a section $\smallS_{NS}: \cN\to\cA$ of the affine bundle $\cA$, that determines an isomorphism $\imath_{NS}:\cA\simeq T^{*}\cN$ by 
$$\cA(E)\ni d+A\mapsto A-A_{E}\in T^{*}_{E}\cN.$$

On the space
$\Gamma(\cN,\cA)$ of smooth sections of the bundle $\cA$ there is a naturally defined operator $\bar\del: \Gamma(\cN,\cA)\to \Omega^{1,1}(\cN)$.  Namely,  over a coordinate chart $U$
define $\delb \smallS$ as $\delb(\smallS-\smallS_{0})\in\Omega^{1,1}(U)$,
where $\smallS_{0}$ is a holomorphic section of $\mathcal{A}$ over $U$. Clearly, this definition does not depend on the choice of section $\smallS_{0}$.
According to Theorem 1 in \cite{ZT},
\begin{equation} \label{bar-NS}
\bar\del \smallS_{NS}=-2\sqrt{-1}\,\omega_{NAB},
\end{equation}
which shows that $\imath_{NS}$ is not a complex-analytic isomorphism. 
\begin{remark}
Here the pairing between $\cH^{1,0}(X,\mathrm{End}\, E_{\rho})$ and $\cH^{0,1}(X,\mathrm{End}\, E_{\rho})$ is given by
$$\sqrt{-1}\int_{X}\theta\wedge\mu,\quad \theta\in\cH^{1,0}(X,\mathrm{End}\, E_{\rho}),\;\;\mu\in\cH^{0,1}(X,\mathrm{End}\, E_{\rho}),$$
and differs by $\sqrt{-1}$ from the pairing in \cite{ZT}.
\end{remark}
The bundle $\cA\to\cN$ is an analogue of the affine bundle $\cP_{g}\to T_{g}$ of projective connections over the Teichm\"{u}ller space $T_{g}$ of compact Riemann surfaces of genus $g>1$. Bers simultaneous uniformization theorem naturally introduces quasi-Fuchsian projective connections and determines a family of global holomorphic sections $T_{g}\to\cP_{g}$, 
parameterized by the points in $T_{g}$ (see \cite{McM, TT} and references therein). 

However, the bundle $\cA\to\cN$ has no global holomorphic sections, a very natural statement, that we were unable to find in the existing literature. In case 
$\cN=N(n,d)$, the moduli spaces of stable bundles of coprime rank $n$ and degree $d$, analogous statement (where $\cA$ is a bundle of constant central curvature connections) easily follows from
\eqref{bar-NS}. Indeed, if $\smallS$ is a holomorphic section of $\cA\to\cN$, then $\omega_{NAB}=\delb\theta$, where $\theta=\frac{\sqrt{-1}}{2}(\smallS_{NS}-\smallS)$ is a $(1,0)$-form on $\cN$. Since in this case $\cN$ is compact and K\"{a}hler, $\omega_{NAB}$ is a zero class in $H^{2}(\cN,\RR)$, which is obviously a contradiction\footnote{According to Atiyah and Bott \cite{AB}, $\omega_{NAB}$ is a generator of $H^{2}(\cN,\RR)$ in case $(n,d)=1$.}. When $(n,d)\neq 1$ the moduli space $N(n,d)$ needs to be compactified and the geometry is more complicated, so we will not discuss this case.

Instead, we show how to use the  construction in \S\,\ref{holo-rep} to define holomorphic sections over each coordinate chart
in $\cN$.

Namely, for a coordinate chart $U$ centered at $\{E\}$ use the representation $E\simeq E_{\si}$, where $\si$ is a holonomy of some $d+A\in\cA(E)$ (see Sect. \ref{RH}),  and realize each bundle in $U$ as a quotient bundle $E_{\si_{\mu}}$, where $\mu\in H^{0,1}_{\mathrm{dR}}(X,\mathrm{End}\,E_{\si})$. 
Let $d+A_{\si_{\mu}}$ be a connection in $E_{\si_{\mu}}$, associated with the connection $d+0$ in the trivial bundle $\HH\times\CC^{n}\to\HH$. 
The family $\{d+A_{\si_\mu}\}$ 
determines a holomorphic section $\smallS_{\si}$ of the affine bundle $\cA\to\cN$ over $U$.

In analogy with the Teichm\"{u}ller theory, when corresponding projective connections come
from the quasi-Fuchsian uniformization of Riemann surfaces, we call connections $\{d+A_{\si_{\mu}}\}$ \emph{quasi-unitary}. 

For a smooth map $f:\HH\to\mathrm{GL}(n,\CC)$ introduce a notation
$$\mathcal{A}(f)=f^{-1}(z)\frac{\del f}{\del z}(z)dz.$$ 
The operator $\mathcal{A}$ is analogous to the Schwarzian derivative for the case of vector bundles. It satisfies  a vector bundle analog of the Cayley identity,
\begin{equation}\label{caley}
\mathcal{A}(gh)=\mathrm{Ad}\,h^{-1}\cdot\mathcal{A}(g)+\mathcal{A}(h),
\end{equation}
where at least one of the functions $g, h$ is holomorphic.

The following simple result describes the difference between Narasimhan-Seshadri
and quasi-unitary connections as an explicit $(1,0)$-form on $U$. 

\begin{lemma}\label{NS-qu} The $(1,0)$-form $\smallS_{NS}-\smallS_{\si}$ on $U$ is given by
$$\left.(\smallS_{NS}-\smallS_{\si})\right|_{E_{\si_{\mu}}}=\mathcal{A}(g^{\mu})
\in  \cH^{1,0}(X, \mathrm{End}\,E_{\si_{\mu}}),\quad \{E_{\si_{\mu}}\}\in U.$$
\end{lemma}
\begin{proof} The Narasimhan-Seshadri connection in $E_{\si_{\mu}}$ 
is a pullback of the connection $d+0$ in the trivial bundle $\HH\times\CC^{n}\to\CC^{n}$ by the map $g^{\mu}$, and is $d+\mathcal{A}(g^{\mu})$. The quasi-unitary connection in $E_{\si_{\mu}}$ is induced by the connection $d+0$ in the trivial bundle $\HH\times\CC^{n}\to\CC^{n}$,  so  the difference $\smallS_{NS}- \smallS_{\si}$ over $E_{\si_{\mu}}$ is $\mathcal{A}(g^{\mu})\in T^{*}_{\{E_{\si_{\mu}}\}}\cN$.
\end{proof}

\subsection{Quasi-unitary reciprocity} \label{reciprocity} Our next result is a vector bundle analogue of the quasi-Fuchsian reciprocity: the difference between Fuchsian and quasi-Fuchsian projective structures is a $\del$-closed $(1,0)$-form on the  Teichm\"{u}ller space \cite{McM, TT}. Namely, we have the following statement.
\begin{proposition}[Quasi-unitary reciprocity]\label{qu-rec} Let $\smallS_{\si}$ be a quasi-unitary holomorphic section of the bundle $\cA\to\cN$ over a coordinate chart $U$. Then
$$\del(\smallS_{NS}-\smallS_{\si})=0\quad\text{and}\quad \delb(\smallS_{NS}-\smallS_{\si})=-2\sqrt{-1}\,\omega_{NAB}. $$
\end{proposition}
\begin{proof}
Let $\vartheta=\smallS_{NS}-\smallS_{\si}\in\Omega^{1,0}(U)$, where $U$ is a coordinate chart centered at $E_{\rho}\simeq E_{\si}$.
Since each $\{E\}\in U$ has a coordinate chart centered at $E$, it follows from the discussion in \S\,\ref{c-coordinates} and \S\,\ref{holo-rep}, that it is sufficient to prove that
$$\del\vartheta(\mu,\nu)=0$$
for all $\mu,\nu\in H^{0,1}_{\mathrm{dR}}(X,\mathrm{End}\,E_{\si})$. Denote by $L_{\mu}=\dfrac{\del}{\del\vep_{\mu}}$ and $L_{\nu}=\dfrac{\del}{\del\vep_{\nu}}$ corresponding vector fields on $U$.
Since $[L_{\mu}, L_{\nu}]=0$,
we have by Cartan formula 
$$\del\vartheta\left(L_{\mu},L_{\nu}\right)=L_{\mu}\vartheta(L_{\nu})-L_{\nu}\vartheta(L_{\mu}).$$

Now consider commutative diagram \eqref{cd-1} with $\vep\mu$ instead of $\mu$. Using Lemma \ref{NS-qu}, we have at $\{E_{\si_{\vep\mu}}\}\in U$,
\begin{align*}
\vartheta(L_{\nu}) &=\sqrt{-1}\int_{X}\mathcal{A}(g^{\vep\mu})\wedge \mathrm{Ad}\,f_{\vep\mu}(\nu)
=\sqrt{-1}\int_{X}\mathrm{Ad}\,f^{-1}_{\vep\mu}\cdot \mathcal{A}(g^{\vep\mu})\wedge \nu.
\end{align*}
It follows from  \eqref{cd-1} and  \eqref{caley} that
\begin{equation}\label{main-identity}
\mathrm{Ad}\,f^{-1}_{\vep\mu}\cdot \mathcal{A}(g^{\vep\mu})=\mathrm{Ad}\, g^{-1}\cdot\mathcal{A}(f^{\vep\hat\mu})+\mathcal{A}(g)-\mathcal{A}(f_{\vep\mu}).
\end{equation}
Differentiate, in a complex-analytic sense, equation \eqref{main-identity} with respect to $\vep$ and set $\vep=0$. Since $f_{\vep\mu}(z)$ is holomorphic in $\vep$ and antiholomorphic in $z$,  
we get
$$\frac{\del}{\del z}\left(\left.\frac{\del}{\del\vep}\right|_{\vep=0}f_{\vep\mu}(z)\right)=0,$$
so
$$\left.\frac{\del}{\del\vep}\right|_{\vep=0}\left(\mathrm{Ad}\,f^{-1}_{\vep\mu}\cdot \mathcal{A}(g^{\vep\mu})\right)=\mathrm{Ad}\,g^{-1}\left(\frac{\del \dot{f}^{\hat\mu}_{+}}{\del z}\right).$$
Thus we obtain
\begin{align*}
L_{\mu}\vartheta\left(L_{\nu}\right) &=\sqrt{-1}\int_{X}\mathrm{Ad}\,g^{-1}\left(\frac{\del \dot{f}^{\hat\mu}_{+}}{\del z}\right)
dz\wedge \nu=\sqrt{-1}\int_{X}\frac{\del \dot{f}^{\hat\mu}_{+}}{\del z}
dz\wedge \hat\nu\\
&=\sqrt{-1}\iint_{F}\tr\left\{\frac{\del \dot{f}^{\hat\mu}_{+}}{\del z}(z)
\hat\nu(z)\right\} dz\wedge d\bar{z}.
\end{align*}
Using Hodge decomposition
$$\hat{\mu}=P(\hat\mu)+\delb F,\quad \hat{\nu}=P(\hat\nu)+\delb G,$$
where $P(\hat\mu), P(\hat\nu) \in \cH^{0,1}(X,\mathrm{End}\,E_{\rho})$ and $F, G\in \Omega^{0}(X,\mathrm{End}\,E_{\rho})$, and taking into account Corollary \ref{cor-1}, we obtain
$$\dot{f}^{\hat\mu}_{+}=\cE+F,$$
where $\cE$ is the Eichler integral for  $P(\hat\mu)$. Since $\cE$ is antiholomorphic and $F$ and $G$ are $\Gamma$-automorphic functions with representation $\mathrm{Ad}\,\rho$,  by Stokes theorem we have
\begin{align*}
L_{\mu}\vartheta(L_{\nu}) &=\sqrt{-1}\iint_{F}\tr\left\{\frac{\del F}{\del z}(z)\left(P(\hat\nu)+\frac{\del G}{\del\bar{z}}(z)\right)\right\}dz\wedge d\bar{z}\\
&=\sqrt{-1}\iint_{F}\tr\left\{\frac{\del F}{\del z}(z)\frac{\del G}{\del\bar{z}}(z)\right\}dz\wedge d\bar{z}\\
&=\sqrt{-1}\iint_{F}\tr\left\{\frac{\del G}{\del z}(z)\frac{\del F}{\del\bar{z}}(z)\right\}dz\wedge d\bar{z}=L_{\nu}\vartheta(L_\mu). 
\end{align*}

For convenience of the reader here is the proof of the second statement, which is Theorem 1 in \cite{ZT}. Since $\vartheta$ is a $(1,0)$-form and $[L_{\nu},L_{\bar\mu}]=0$,
Cartan formula simplifies
$$\delb\vartheta\left(L_{\nu},L_{\bar\mu}\right)=-L_{\bar\mu}\vartheta(L_{\nu}).$$
Differentiating, in a complex-analytic sense, equation \eqref{main-identity} with respect to $\bar\vep$ and setting $\vep=0$, we obtain
$$\left.\frac{\del}{\del\bar\vep}\right|_{\vep=0}\left(\mathrm{Ad}\,f^{-1}_{\vep\mu}\cdot \mathcal{A}(g^{\vep\mu})\right)=\mathrm{Ad}\,g^{-1}\left(\frac{\del \dot{f}^{\hat\mu}_{-}}{\del z}\right).$$
It follows from the Hodge decomposition and Corollary \ref{cor-1} that $\dot{f}_{-}^{\hat\mu}=-\cE^{*}$, so 
\begin{align*}
\delb\vartheta\left(L_{\nu},L_{\bar\mu}\right)&=\sqrt{-1}\int_{X}\mathrm{Ad}\,g^{-1}\cdot P(\hat\mu)^{*}\wedge\nu=\sqrt{-1}\int_{X}P(\hat\mu)^{*}\wedge\hat\nu\\
&=\la P(\hat\nu),P(\hat\mu)\ra,
\end{align*}
and it follows from \eqref{anb-complex} that $\delb\vartheta=-2\sqrt{-1}\,\omega_{NAB}$.
\end{proof}

\begin{remark} Formula $L_{\mu}\vartheta(L_{\nu})=L_{\nu}\vartheta(L_\mu)$ is a precise analogue of the quasi-Fuchsian reciprocity (see \cite[Theorem 6.1]{McM} and \cite[Proposition 4.1]{TT}). 
Our commutative diagram \eqref{cd-1} is analogous to the commutative diagram (3.2) in \cite{TT}. However, for weight 2 automorphic forms we no longer can use series over the group $\Gamma$,
as in \cite{McM, TT} for the forms of weight 4, so we replace this argument by using the Hodge decomposition.
\end{remark}

\section{Liouville form on $\cA$}\label{Liouville}

\subsection{Holomorphic symplectic form}\label{hol symplectic}
The total space of the holomorphic cotangent bundle $T^{*}M$ of a complex manifold $M$ carries a natural holomorphic symplectic $(2,0)$-form $\omega_{L}$ --- the Liouville symplectic form.
Namely, $\omega_{L}=d\theta_{L}$, where the holomorphic $(1,0)$-form $\theta_{L}$ --- the Liouville 1-form --- is defined by
$$\theta_{L}(u)=p(\pi_{*}u),\quad \text{where}\quad u\in T_{(p,q)}T^{*}M,\;\; p\in T^{*}_{q}M,$$
and $\pi:T^{*}M\to M$ is the canonical projection.
It is convenient to use on $T^{*}U$ local coordinates 
$$(\bm{p} ,\bm{q}) =(p_{1},\dots,p_{n}, q^{1},\dots,q^{n}),$$
where  $\bm{q} =(q^{1},\dots,q^{n})$ --- are complex coordinates on  $U\subset M$, and complex coordinates $\bm{p} =(p_{1},\dots,p_{n})$ correspond to the basis $dq^{1},\dots, dq^{n}$ in $T^{*}_{q}M$ for $q\in U$. In this notation, commonly used by physicists, we have
$$\theta_{L}=\bm{p}d\bm{q}=\sum_{k=1}^{n}p_{k}dq^{k}\quad\text{and}\quad \omega_{L}=d\bm{p}\wedge d\bm{q}=\sum_{k=1}^{n}dp_{k}\wedge dq^{k}. $$

The following statement is trivial, but rather useful (see, e.g., \cite{Loustau}).
\begin{lemma}\label{taut}
Let $\alpha$ be a $(1,0)$-form on $M$, considered as a section of $\pi: T^{*}M\to M$. Then 
$$\alpha^{*}(\theta_{L})=\alpha\quad\text{and}\quad \alpha^{*}(\omega_{L})=d\alpha.$$
\end{lemma}
\begin{proof}
In physics notation, this is a tautology. Namely, in local coordinates
$$\alpha=\bm{a}(\bm{q})d\bm{q}=\sum_{k=1}^{n}a_{i}(q^{1},\dots,q^{n})dq^{i},$$
so $M\ni q\mapsto \alpha_{q}=\bm{a}(\bm{q})\in T^{*}_{q}M$ 
and $\alpha^{*}(\theta_{L})=\alpha$. A formal argument is also trivial: since $\pi\circ\alpha=\id_{M}$, then for $v\in T_{q}M$ we have $\alpha_{\ast}(v)\in T_{(\alpha_{q},q)}T^{\ast}M$ and
$$\alpha^{\ast}(\theta)(v)=\theta(\alpha_{\ast}(v))=\alpha_{q}(\pi_{\ast}(\alpha_{\ast} v))=\alpha(v).\qedhere$$
\end{proof}

\subsection{Pullback of Liouville form} \label{L-A} Here we apply results in \S\,\ref{hol symplectic} to the case $M=\cN$. The Liouville form $\omega_{L}$ at a point $(\theta,E)\in T^{*}\cN$, where $\theta\in\cH^{1,0}(X,\mathrm{End}\,E)$, evaluated at the tangent vectors $(\theta_{1},\mu_{1}), (\theta_{2},\mu_{2})\in T_{(\theta,E)}T^{*}\cN$, 
can be written as
\begin{equation}\label{L-N}
\omega_{L}((\theta_{1},\mu_{1}),(\theta_{2},\mu_{2}))
=\frac{\sqrt{-1}}{2}\int_{X}(\theta_{1}\wedge\mu_{2}-\theta_{2}\wedge\mu_{1}),
\end{equation}
where $\mu_{1}, \mu_{2}\in \cH^{0,1}(X,\mathrm{End}\,E)$ and $\theta_{1},\theta_{2}\in\cH^{1,0}(X,\mathrm{End}\,E)$.

As discussed in \S\,\ref{hol sections},
every smooth section $\smallS: \cN\to\cA$ determines an isomorphism $\imath_{\smallS}:\cA\xrightarrow{\sim} T^{\ast}\cN$ by
$$\cA(E)\ni A\mapsto \imath_{\smallS}(A)=A-\smallS(E)\in T^{*}_{E}\cN$$
and turns $\cA$ into a symplectic manifold with the symplectic form  
$\imath_{\smallS}^{*}(\omega_{L})$. If the section $\smallS$ was holomorphic, $\imath_{\smallS}^{*}(\omega_{L})$ would be a holomorphic symplectic form on $\cA$, but the bundle $\cA\to\cN$ has no global holomorphic sections.

Nevertheless, one can use local holomorphic sections to pull back the Liouville symplectic form $\omega_{L}$ to $\cA$. The following simple statement (cf. \cite{Loustau})
guarantees when such local symplectic forms on $\cA$ agree.
\begin{lemma}\label{unique}
Let $\smallS_{1}$ and $\smallS_{2}$ be holomorphic sections of $\cA\to\cN$ over an open $U\subset \cN$. The equality $\imath_{\smallS_{1}}^{*}(\omega_{L})=\imath_{\smallS_{2}}^{*}(\omega_{L})$ on $\cA_{U}$ --- a restriction of $\cA$ over $U$ ---
is equivalent to the condition $\pi^{*}(\del(\smallS_{1}-\smallS_{2}))=0$, where $\pi: T^{*}\cN\to\cN$ is the canonical projection. 
\end{lemma}
\begin{proof} We have for $A\in\cA(E)$, $\{E\}\in U$,
\begin{align*}
 (\imath_{\smallS_{1}}- \imath_{\smallS_{2}})(A) & =A-\smallS_{1}(E)-(A-\smallS_{2}(E)) \\ 
 &= \smallS_{2}(E)-\smallS_{1}(E),
\end{align*}
so $\imath_{\smallS_{1}}- \imath_{\smallS_{2}}=(\smallS_{2}-\smallS_{1})\circ \pi$.  It follows from Lemma \ref{taut} that
 $$\imath_{\smallS_{1}}^{*}(\theta_{L})-\imath_{\smallS_{2}}^{*}(\theta_{L})=\pi^{*}((\smallS_{2}-\smallS_{1})^{*}\theta_{L})=\pi^{*}(\smallS_{2}-\smallS_{1}).$$
and
$$\imath_{\smallS_{1}}^{*}(\omega_{L})-\imath_{\smallS_{2}}^{*}(\omega_{L})=\pi^{*}(d(\smallS_{2}-\smallS_{1}))=\pi^{*}(\del(\smallS_{1}-\smallS_{2})).\qedhere$$
\end{proof}

Now over each coordinate chart $U$ use a quasi-unitary section $\smallS_{\si}$ of the bundle $\cA\to\cN$ to define a holomorphic $(2,0)$ symplectic form  
$$\omega_{U,\si}=\imath_{\smallS_{\si}}^{*}(\omega_{L})\quad\text{over}\quad\cA_{U}.$$
We have the following result.

\begin{proposition}\label{independence}
Holomorphic $(2,0)$-forms $\omega_{U,\si}$ determine a holomorphic symplectic form $\omega$ on $\cA$. The form $\omega$ does not depend on the choices of local quasi-unitary sections.
\end{proposition}
\begin{proof} We need to show that $\omega_{U_{1},\si_{1}}=\omega_{U_{2},\si_{2}}$ on $U_{1}\cap U_{2}\neq\emptyset$.  By Proposition \ref{qu-rec}, we have 
$$\del(\smallS_{NS}-\smallS_{\si_{1}})=0\quad\text{and}\quad \del(\smallS_{NS}-\smallS_{\si_{2}})=0$$
in coordinate charts $U_{1}$ and $U_{2}$, so that $\del(\smallS_{\si_{1}}-\smallS_{\si_{2}})=0$ on $U_{1}\cap U_{2}$.  Now the result follows form
Lemma \ref{unique}.
\end{proof}

\begin{remark}\label{pull-back-smooth} Narasimhan-Seshadri and quasi-unitary connections determine local non-holomorphic sections $\jmath_{\si}=\smallS_{NS}- \smallS_{\si} : U\to T^{*}U$ of the bundle $T^{*}\cN\to\cN$.
Explicitly,
$$\jmath_{\si}(E_{\si_{\mu}})=\mathcal{A}(g^{\mu})=\mathcal{A}(f^{\hat\mu}g f^{-1}_{\mu})$$ 
(cf. with the global section of the bundle $T^{*}T_{g}\to T_{g}$ in \cite[Remark 4]{Tak}). It follows from Proposition \ref{qu-rec} and Lemma \ref{taut} that
$$\jmath^{*}_{\si}(\omega_{L})=-2\sqrt{-1}\,\omega_{NAB}.$$
\end{remark}
\section{Goldman form on $\cA$}\label{final}

Let $\nabla(t)$ be a smooth curve in the complex manifold $\cA$ of zero curvature connections, $\nabla(0)=\nabla\in\cA(E)$ for some $\{E\}\in\cN$, and let $v=\dot\nabla$ be its tangent vector at $t=0$, considered as a vector in the holomorphic tangent space to $\cA$ at $\nabla$. Let $\si=\cQ(\nabla)$ be the holonomy of the connection $\nabla$ in $E$. Under the Riemann-Hilbert correspondence, we have a curve $\cQ(\nabla(t))$ in the character variety $\cK_{\CC}$ with the holomorphic tangent vector $\cQ_{*}(v)\in T_{\si}\cK_{\CC}$ at $t=0$. 

On the other hand, over a neighborhood $U$ we have an isomorphism $\imath=\imath_{\smallS_{\si}}: \cA_{U}\xrightarrow{\sim} T^{*}U$ (see \S\S\,\ref{holo-rep}-\ref{hol sections} and \S\,\ref{L-A}) and we can consider, for small $t$,
the image of the curve $\nabla(t)$ in $T^{*}U$ with the holomorphic tangent vector $\imath_{*}(v)$. 
According to Proposition \ref{independence}, the pullback of the Liouville  form $\omega_{L}$ to $\cA_{U}$
does not depend on a choice of the isomorphism $\imath$.

\subsection{Differential of the Riemann-Hilbert correspondence} \label{differential} Here we explicitly describe the differential of the map $hol=\cQ\circ\imath^{-1}: T^{*}U\to\cK_{\CC}$ for $\imath=\imath_{\smallS_{\si}}$ corresponding to the isomorphism $E\simeq E_{\si}$, as in \S\S\,\ref{holo-rep}-\ref{hol sections}. 

Namely, let  $E(t)$ be the projection of the curve $\nabla(t)$ in $\cA$ to $U$.  Each bundle in $U$ is realized as $E_{\si_{\mu}}$, and the curve 
$E(t)$ at $t=0$ has a tangent vector $\mu\in H_{\mathrm{DR}}^{0,1}(X,\mathrm{End}\,E_{\si})$. Under the isomorphism $\imath: \cA_{U}\xrightarrow{\sim} T^{*}U$ 
we have 
$$\nabla(t)=d+A_{t}, \quad\text{where}\quad A_{t}\in H_{\mathrm{dR}}^{1,0}(X,\mathrm{End}\,E_{\si_{t\mu}}),\quad A_{0}=0,$$ 
and without changing the tangent vector $v=\dot\nabla$ we can assume that $E(t)=E_{\si_{t\mu}}$. 
Put 
$$B_{t}=\mathrm{Ad}f^{-1}_{t\mu}\cdot A_{t}\in \Omega^{1,0}(X,\mathrm{End}\,E_{\si}).$$
Then
$$T_{\nabla}\cA\ni v\mapsto \imath_{*}(v)=(P(\dot{B}),\mu)\in T_{(0,E_{\si})}(T^{*}\cN),$$
where
$$\dot{B}=\left.\frac{d}{dt}\right|_{t=0}B_{t} \in \Omega^{1,0}(X,\mathrm{End}\,E_{\si}),$$
and $P$ is the projection operator onto $H_{\mathrm{dR}}^{1,0}(X,\mathrm{End}\,E_{\si})$.

The image of the curve $A_{t}$ under the map $hol=\cQ\circ\imath^{-1}$ is a curve $\si_{t}$ in the character variety $\cK_{\CC}$  --- a family of irreducible representations $\si_{t}:\Gamma\to\mathrm{GL}(n,\CC)$, defined by 
$$g_{t}(\g z)=\si_{t\mu}(\g)g_{t}(z)\si_{t}(\g)^{-1},$$
where the functions $g_{t}$ satisfy the parallel transport equation
\begin{equation}\label{transport-1}
\frac{dg_{t}}{dz}(z)+A_{t}(z)g_{t}(z)=0,
\end{equation}
and $A_{t}=A_{t}(z)dz$.
The holomorphic tangent vector to $\si_{t}$ at $t=0$ is $\chi\in H^{1}(\Gamma,\frak{g}_{\mathrm{Ad}\,\si})$, where
\begin{equation}\label{chi-curve}
\chi(\g)=\left.\frac{d}{dt}\right|_{t=0}\si_{t}(\gamma)\si(\g)^{-1},
\end{equation}
so
$$hol_{*}(P(\dot{B}),\mu)=\chi. $$

Cocycle $\chi$ admits the following expression in terms of $\dot{B}$ and $\mu$. 
Put $h_{t}(z)=f_{t\mu}^{-1}(z)g_{t}(z)$. It follows from \eqref{ode} and \eqref{transport-1} that the functions $h_{t}(z)$ satisfy a system of partial differential equations
\begin{align*}
\frac{\del h_{t}}{\del z}(z) &=-B_{t}(z)h_{t}(z),\\
\frac{\del h_{t}}{\del\bar{z}}(z) &=-t\mu(z) h_{t}(z),
\end{align*}
where $B_{t}=B_{t}(z)dz$, 
and have the property
\begin{equation}\label{h-t}
h_{t}(\g z)=\si(\g)h_{t}(z)\sigma_{t}(\g)^{-1}.
\end{equation}
Differentiating this system with respect to $t$ at $t=0$ and putting
$$\dot{h}(z)=\left.\frac{d}{dt}\right|_{t=0}h_{t}(z),$$
we get the system
\begin{align}
\frac{\del\dot{h}}{\del z}(z) & =-\dot{B}(z),\label{h-1}\\
\frac{\del \dot{h}}{\del\bar{z}}(z) &=-\mu(z),\label{h-2}
\end{align}
and it follows from \eqref{h-t} that
\begin{equation}\label{h-dot}
\chi(\g)=\si(\gamma)\dot{h}(z)\si(\g)^{-1}-\dot{h}(\g z).
\end{equation}

Thus we expressed $\chi=hol_{*}(P(\dot{B}),\mu)$ in terms of the solution of the system \eqref{h-1}--\eqref{h-2}.

\subsection{Pullback of Goldman form} \label{pullback} Here we prove our main result.
\begin{theorem}\label{main}
The pullback to $\cA$ of the Goldman symplectic form $\omega_{G}$ on $\cK_{\CC}$ by the Riemann-Hilbert correspondence is $-2\sqrt{-1}$ times the pullback to $\cA$of the Liouville symplectic form
$\omega_{L}$ on $T^{*}\cN$  by the quasi-unitary connections.
\end{theorem}
\begin{proof} Let $v_{1}$ and $v_{2}$ be holomorphic tangent vectors to $\cA$ at a point $A\in\cA(E)$, zero curvature connection in $E$ with the holonomy $\si\in\cK_{\CC}$, and let
$\imath_{*}(v_{1})=(P(\dot{B}_{1}),\mu_{1})$ and $\imath_{*}(v_{2})=(P(\dot{B}_{2}),\mu_{2})$. We have, using \eqref{L-N} and equations \eqref{h-1}--\eqref{h-2},
\begin{align*}
\omega_{L}(\imath_{*}(v_{1}),\imath_{*}(v_{2})) & =\frac{\sqrt{-1}}{2}\int_{X}\left\{P(\dot{B}_{1})\wedge\mu_{2}-P(\dot{B}_{2})\wedge\mu_{1}\right\}\\
& =\frac{\sqrt{-1}}{2}\iint_{F}\tr\left(\dot{B}_{1}(z)\mu_{2}(z)-\dot{B}_{2}(z)\mu_{1}(z)\right)dz\wedge d\bar{z}\\
& =\frac{\sqrt{-1}}{2}\iint_{F}\tr\left(\frac{\del\dot{h}_{1}}{\del z}(z)\frac{\del\dot{h}_{2}}{\del\bar{z}}(z)-\frac{\del\dot{h}_{2}}{\del z}(z)\frac{\del\dot{h}_{1}}{\del\bar{z}}(z)\right)dz\wedge d\bar{z}\\
& =\frac{\sqrt{-1}}{2}\iint_{F}\tr (d\dot{h}_{1}\wedge d\dot{h}_{2}) =\frac{\sqrt{-1}}{2}\int_{\del F}\tr( \dot{h}_{1}d\dot{h}_{2}).
\end{align*}
As in the proof of Theorem \ref{th: E-S} we continue, using \eqref{f-domain} and \eqref{h-dot},
\begin{align*}
\int_{\del F}\tr( \dot{h}_{1}d\dot{h}_{2}) &=\sum_{i=1}^{2g}\left\{\int_{S_{i}}\tr( \dot{h}_{1}d\dot{h}_{2}) -\int_{\lambda_{i}(S_{i})}\tr( \dot{h}_{1}d\dot{h}_{2})\right\}\\
& =\sum_{i=1}^{2g}\int_{S_{i}}\tr\left(\chi_{1}(\lambda_{i})\, \mathrm{Ad}\,\si(\lambda_{i})\cdot d\dot{h}_{2}\right) \\
&=-\sum_{i=1}^{2g}\tr\left(\chi_{1}(\lambda^{-1}_{i})\int_{S_{i}}d\dot{h}_{2}\right).
\end{align*}
Using again  \eqref{h-dot}, we get
\begin{align*}
\int_{S_{k}}d\dot{h}_{2} &
=\int_{R_{k-1}z_{0}}^{R_{k-1}a_{k}z_{0}}d\dot{h}_{2}=\mathrm{Ad}\,\si(R_{k-1})\cdot\int_{z_{0}}^{a_{k}z_{0}}d\dot{h}_{2}\\
& =-\mathrm{Ad}\,\si(R_{k-1})\cdot\chi_{2}(a_{k}) +(\mathrm{Ad}\,\si(R_{k-1}a_{k})-\mathrm{Ad}\,\si(R_{k-1}))\cdot \dot{h}_{2}(z_{0})
\end{align*}
and
\begin{align*}
\int_{S_{k+g}}d\dot{h}_{2} &=-\int_{R_{k}z_{0}}^{R_{k}b_{k}z_{0}}d\dot{h}_{2}=-\mathrm{Ad}\,\si(R_{k})\cdot\int_{z_{0}}^{b_{k}z_{0}}d\dot{h}_{2}\\
& =\mathrm{Ad}\,\si(R_{k})\cdot\chi_{2}(b_{k}) -(\mathrm{Ad}\,\si(R_{k}b_{k})-\mathrm{Ad}\,\si(R_{k}))\cdot \dot{h}_{2}(z_{0}),
\end{align*}
Using \eqref{S-lambda} and \eqref{par-1}, we obtain
\begin{gather*}
-\sum_{k=1}^{g}\tr\left\{\chi_{1}(\lambda^{-1}_{k})\int_{S_{k}}d\dot{h}_{2}\right\}=\sum_{k=1}^{g}\tr\left(\chi_{1}(\alpha_{k})\,\mathrm{Ad}\,\si(R_{k-1})\cdot\chi_{2}(a_{k})\right) +I_{1}
\end{gather*}
where
\begin{gather*}
I_{1}=-\sum_{k=1}^{g}\tr\left\{\left(\chi_{1}(a_{k}^{-1}R_{k-1}^{-1}\alpha_{k})-\chi_{1}(a_{k}^{-1}R_{k-1}^{-1})-\chi_{1}(R_{k-1}^{-1}\alpha_{k})+\chi_{1}(R_{k-1}^{-1})\right)\dot{h}_{2}(z_{0})\right\}.
\end{gather*}
Similarly,
\begin{gather*}
-\sum_{k=1}^{g}\tr\left\{\chi_{1}(\lambda^{-1}_{k})\int_{S_{k+g}}d\dot{h}_{2}\right\}=-\sum_{k=1}^{g}\tr\left(\chi_{1}(\beta_{k})\,\mathrm{Ad}\,\si(R_{k})\cdot\chi_{2}(b_{k})\right) +I_{2},
\end{gather*}
where
\begin{gather*}
I_{2}=\sum_{k=1}^{g}\tr\left\{\left(\chi_{1}(b_{k}^{-1}R_{k}^{-1}\beta_{k})-\chi_{1}(b_{k}^{-1}R_{k}^{-1})-\chi_{1}(R_{k}^{-1}\beta_{k})+\chi_{1}(R_{k}^{-1})\right)\dot{h}_{2}(z_{0})\right\}.
\end{gather*}
Finally, using  \eqref{dual-gen}--\eqref{gen-dual} and relations $R_{k}b_{k}a_{k}=R_{k-1}a_{k}b_{k}$, we obtain
\begin{gather*}
I_{1}+I_{2}=\sum_{k=1}^{g}\tr\Big\{\left(\chi_{1}(a_{k}^{-1}R_{k-1}^{-1})-\chi_{1}(a_{k}^{-1}b_{k}^{-1}R_{k}^{-1})+\chi_{1}(b_{k}^{-1}R_{k}^{-1})-\chi_{1}(R_{k-1}^{-1})\right.\\
\left.\left.+\chi_{1}(b_{k}^{-1}a_{k}^{-1}R_{k-1}^{-1})-\chi_{1}(b_{k}^{-1}R_{k}^{-1})-\chi_{1}(a^{-1}_{k}R_{k-1}^{-1})+\chi_{1}(R_{k}^{-1})\right)\dot{h}_{2}(z_{0})\right\} =0,
\end{gather*}
so by \eqref{dual}
\begin{gather*}
\int_{\del F}\tr( h_{1}dh_{2}) = \frac{\sqrt{-1}}{2}\omega_{G}(\chi_{1},\chi_{2}).\qedhere
\end{gather*}
\end{proof}

\begin{remark}The proof of Theorem \ref{main} is a simplified version of the proof of Theorem 1 in \cite{Tak}, adapted for the case of vector bundles.

\end{remark}
\begin{remark} Theorem \ref{main} is analogue of `Kawai theorem' --- a theorem in \cite{Kawai} that the pullback by the monodromy map of the Goldman symplectic form on the $\mathrm{PSL}(2,\CC)$-character variety to the bundle of projective structures on the Riemann surfaces, identified with the holomorphic cotangent bundle to the Teichm\"{u}ller space by the Bers section, coincides with the Liouville symplectic form. However,  the proof of formula (5.4) in \cite{Kawai}  is not
correct since the functions $\dot{w}^{\nu}(z)$ on $\HH$ (using Ahlfors notation as in \cite{Kawai, Tak}) do not transform like vector fields under the group $\Gamma$. For harmonic $\nu$ formula (5.4) immediately follows from Ahlfors formula that $\dot{w}^{\nu}_{zzz}(z)=0$ and (5.4) trivially holds, while for general $\nu$ 
it is precisely the quasi-Fuchsian reciprocity, proved in \cite{McM,TT}. We refer the reader to \cite{Tak} 
for the proof of `Kawai theorem' in the spirit of Riemann bilinear relations. 
\end{remark}

The main results, Theorems \ref{E-S-G} and \ref{main} and Proposition \ref{independence},  can be succinctly combined with Remark \ref{pull-back-smooth} in the form of following commutative diagram, where all maps are symplectomorphisms
\begin{equation}\label{sum-diagram}
\begin{tikzcd}
(\cA, -2\sqrt{-1}\,\imath_{\si}^{*}(\omega_{L}))\arrow{r}{\cQ} & (\cK_{\CC},\omega_{G})\\
(\cN, -4\,\omega_{NAB})\arrow{u}{\smallS_{NS}}\arrow{r}{\imath} & (\cK_{\RR},\omega_{G})\arrow[u, hook]
\end{tikzcd}.
\end{equation}
Here $\smallS_{NS}=\imath_{\si}^{-1}\circ\jmath_{\si}$ and $\cK_{\RR}\hookrightarrow\cK_{\CC}$ is the inclusion map.
\subsection{Generalizations} \label{generalization} It is straightforward to generalize the obtained results to the moduli space $\cN(n,d)$ of stable vector bundles of arbitrary rank $n$ and degree $d$ on $X$ (it is sufficient to consider $-n<d\leq 0$). As in \cite{NS2},  we realize $X\simeq \Gamma\bk\HH$ as an orbifold Riemann surface with conical singularity of order $n$ at a given point $x_{0}\in X$, so the Fuchsian group $\Gamma$ contains additional elliptic generator $\g_{0}$ of order $n$ with the fixed point $z_{0}\in\HH$ that projects on $x_{0}$. The Narasimhan-Seshadri theorem says that a stable bundle $E$ of rank $n$ and degree $d$  is obtained from an irreducible representation $\rho:\Gamma\to\mathrm{U}(n)$ such that $\rho(\gamma_{0})=\zeta^{-d}I$, where $\zeta$ is a primitive $n$-th root of unity.  

Correspondingly, a fiber $\cA(E)$ of the bundle $\cA$ is the affine space of connections in $E$ compatible with the holomorphic structure and having a constant central curvature as in \cite{Donaldson}, or equivalently, connections $\nabla$ in $E$ such that their curvature (considered as a current) satisfies
$$\frac{\sqrt{-1}}{2\pi}\int_{X}f\wedge\nabla^{2}=\mu(E)\tr\,\!f(x_{0}),$$
for all $f\in\Omega^{0}(X,\mathrm{End}\,E)$, as in \cite{ZT}. Concretely, realizing $E$ as a quotient bundle $E_{\rho}$, such connections are $d+A$, where $A$ is a meromorphic matrix-valued automorphic form of weight $2$ for $\Gamma$ with the representation $\mathrm{Ad}\,\rho$, having only simple poles at $\Gamma\cdot z_{0}$ with residue $-\mu(E) I$, where $I$ is $n\times n$ identity matrix. Corresponding $\mathrm{GL}(n,\CC)$-character variety consists of equivalence classes of irreducible representations $\si: \Gamma\to\mathrm{GL}(n,\CC)$ such that $\si(\g_{0})=\zeta^{-d}I$.  All results in the paper and their proofs generalize verbatim to this case.

Finally, obtained results can be also generalized to the moduli spaces of parabolic bundles. In this case, one needs to use parabolic cohomology as in \cite{Meneses, Tak}, the matrix-valued cusp forms of weight $2$
as in \cite{TZ2, Meneses}, and define a character variety by specifying conjugacy classes for the parabolic generators of $\Gamma$, as in \cite{GHJW}.



\begin{thebibliography}{99}
\bibitem{Ahlfors} \textit{L.V. Ahlfors}, Some remarks on Teichm\"{u}ller's space of Riemann surfaces, Ann. of Math. \textbf{74}
(1961), 171--191.
\bibitem{Ahlfors-2} \textit{L.V. Ahlfors}, \emph{Lectures on Quasiconformal Mappings}, Wadsworth \& Brooks/Cole Advanced Books \& Software, Monterey, CA, 1987, with the assistance of Clifford J. Earle Jr. Reprint of the 1966 original.
\bibitem{AB} \textit{M.F. Atiyah} and \textit{R. Bott}, The Yang-Mills equations over Riemann surfaces, Phil. Trans. R. Soc. London, \textbf{A308} (1982), 523--615.
\bibitem{Bertola} \textit{M. Bertola}, \textit{C. Norton} and \textit{G. Ruzza}, Higgs fields, non-abelian Cauchy kernels and the Goldman symplectic structure, preprint 2021, \url{http://arxiv.org/abs/2102.09520}.
\bibitem{Donaldson} \textit{S.K. Donaldson},  A new proof of a theorem of Narasimhan and Seshadri, J. Diff. Geometry \textbf{18} (1983), 269--277.
\bibitem{Goldman1} \textit{ W.M. Goldman}, The symplectic nature of fundamental groups of surfaces, Advances in Math.
\textbf{54} (1984), 200--225.
\bibitem{Goldman2}  \textit{W.M. Goldman}, Topological components of spaces of representations, Invent. Math. \textbf{93} (1988), 557--607.
\bibitem{Gunning}  \textit{R.C. Gunning}, Lectures on Vector Bundles over Riemann Surfaces, Math. Notes \textbf{6}, Princeton University Press, Princeton, NJ, 1967.
\bibitem{GHJW}  \textit{K. Guruprasad},  \textit{J. Huebschmann},  \textit{L. Jefferey} and  \textit{A. Weinstein}, Group systems, groupoids and moduli spaces of parabolic bundles, Duke Math. J. \textbf{89} (1997), 377--412.
\bibitem{Hejhal} \textit{D. Hejhal}, The Selberg Trace Formula for $\mathrm{PSL}(2,\RR)$.Volume 2. Lecture Notes in Mathematics, vol. \textbf{1001}, Springer, Berlin, Heidelberg (1983).
\bibitem{Kawai}  \textit{S. Kawai}, The symplectic nature of the space of projective connections on Riemann surfaces, Math. Ann. \textbf{305}:(1) (1996), 161--182.
\bibitem{Krichever} \textit{I. Krichever}, Isomonodromy equations on algebraic curves, canonical transformations and Witham equations, Moscow Math. J. \textbf{4} (2002), 717--752.
\bibitem{Loustau}  \textit{B. Loustau}, The complex symplectic geometry of the deformation space of complex projective structures, Geom. Topol. \textbf{19} (2015) 1737--1775.
\bibitem{McM}  \textit{C.T. McMullen}, The moduli space of Riemann surfaces is K\"{a}hler hyperbolic, Ann. Math. (2) \textbf{151}(1) (2000), 327--357.
\bibitem{Meneses}  \textit{C. Meneses}, On Shimura's isomorphism and $(\Gamma,G)$-bundles on the upper half plane,
 Contributions of Mexican mathematicians abroad in pure and applied mathematics, 
 101--113, Contemp. Math., \textbf{709}, Aportaciones Mat., Amer. Math. Soc., Providence, RI,  2018.
\bibitem{NS1}  \textit{M.S. Narasimhan} and  \textit{C.S. Seshadri}, Holomorphic vector bundles on a compact Riemann surface,
Math. Ann. \textbf{155} (1964), 69--80.
\bibitem{NS2}  \textit{M.S. Narasimhan} and  \textit{C.S. Seshadri}, Stable and unitary vector bundles on a compact Riemann
surface, Ann. of Math. (2) \textbf{82} (1965), 540--567.
\bibitem{Narasimhan}  \textit{M.S. Narasimhan}, Elliptic operators and differential geometry of moduli spaces of
vector bundles on compact Riemann surfaces, In: Proceedings of the International conference
on functional analysis and related topics. Tokyo, 68--71, 1969.
\bibitem{Shimura}  \textit{G. Shimura}, Introduction to the Arithmetic Theory of Automorphic Functions,  Iwanami Shoten Publishers and Princeton University Press, 1971.
\bibitem{TT}  \textit{L.A. Takhtajan} and  \textit{L.P. Teo}, Liouville action and Weil-Petersson metric on deformation spaces, global Kleinian reciprocity and holography, Commun. Math. Phys. \textbf{239} (2003), 183-240.
\bibitem{TZ2}  \textit{L.A. Takhtajan} and  \textit{P. Zograf}, The first Chern form on moduli of parabolic bundles, Math. Ann. \textbf{341} (2008), 113--135.
\bibitem{Tak}  \textit{L.A. Takhtajan}, On Kawai theorem for orbifold Riemann surfaces, Math. Ann. \textbf{375} (2019), 923--947.
\bibitem{ZT}  \textit{P.G. Zograf} and  \textit{L.A. Takhtadzyan}, On the geometry of moduli spaces of vector bundles over a Riemann surface, Izv. Akad. Nauk SSSR Ser. Mat. \textbf{53} (1989), no. 4, 753--770 (Russian), English translation in Math. USSR-Izv. \textbf{35} (1990), no. 1, 83--100.
\end{thebibliography}
\end{document}